\documentclass[oneside,english,reqno]{amsart}

\usepackage{ifmtarg}
\usepackage{amsmath,amssymb,amsthm}
\usepackage[utf8]{inputenc}
\usepackage[T1]{fontenc}
\usepackage{hyperref}
\usepackage{cite}
\usepackage{color}
\usepackage{enumitem}

\hypersetup{
breaklinks=true,
colorlinks=true,
linkcolor=blue,
citecolor=blue,
urlcolor=blue,
}

\newtheorem{theorem}{Theorem}[section]

\newtheorem{lemma}[theorem]{Lemma}
\newtheorem{corollary}[theorem]{Corollary}
\theoremstyle{definition}

\theoremstyle{remark}
\newtheorem{remark}[theorem]{Remark}

\numberwithin{equation}{section}

\newcommand{\R}{\mathbb{R}}  

\newcommand{\Z}{\mathbb{Z}} 
\newcommand{\N}{\mathbb{N}} 
\newcommand{\E}{\mathcal{E}} 
\newcommand{\EP}{\E}
\newcommand{\EPD}{\E^{(p)}_{\delta}}
\newcommand{\F}{\mathcal{F}} 
\newcommand{\ER}{\E^{(p)}_{\varepsilon, \delta, \lambda}} 
\newcommand{\Ell}{\mathcal{L}}

\newcommand{\RZ}{\R / \Z}

\newcommand{\gs}{\nabla_s}
\newcommand{\gt}{\nabla_t}
\newcommand{\fs}{\partial_s}
\newcommand{\ft}{\partial_t}
\newcommand{\fx}{\partial_x}

\newcommand{\g}{f}  
\newcommand{\ka}{\kappa} 
\newcommand{\Le}{\mathcal{L}} 

\newcommand{\tg}{\tilde \g}
\newcommand{\rg}{\g^{\varepsilon,\delta}} 
\newcommand{\trg}{\tg^{\varepsilon, \delta}}
\newcommand{\gn}{\g^{(n)}} 

\DeclareMathOperator{\hol}{h\ddot{o}l}

\usepackage[
initials]{amsrefs} 

\AtBeginDocument{%
   \def\MR#1{}
}

\begin{document}
	
\date{\today}

\title[A regularized gradient flow for the $p$-elastic energy]{\Large Regularized gradient flow for the $p$-elastic energy}

\author{Simon Blatt}
\address{Fachbereich Mathematik, Universität Salzburg, Hellbrunner Strasse 34, 5020 Salzburg, Austria}
\email{simon.blatt@sbg.ac.at}
\urladdr{https://blatt.sbg.ac.at/} 

\author{Christopher Hopper}
\address{Fachbereich Mathematik, Universität Salzburg, Hellbrunner Strasse 34, 5020 Salzburg, Austria}
\email{christopher.hopper@sbg.ac.at}

\author{Nicole Vorderobermeier}
\thanks{The authors acknowledges support by the Austrian Science Fund (FWF), Grant P 29487. Special thanks to Armin Schikorra for pointing the authors to the approach to higher differentiability in \cite{Brasco2017}}
\address{Fachbereich Mathematik, Universität Salzburg, Hellbrunner Strasse 34, 5020 Salzburg, Austria}
\email{nicole.vorderobermeier@sbg.ac.at}
\urladdr{https://uni-salzburg.at/index.php?id=209722}


\begin{abstract}
 We prove long-time existence for the negative $L^2$-gradient flow of the $p$-elastic energy, $p\geq 2$, with an additive positive multiple of the length of the curve. To achieve this result we regularize the energy by adding a small multiple of a higher order energy, namely the square of the $L^2$-norm of the normal gradient of the curvature $\kappa$. Long-time existence is proved for the gradient flow of these new energies together with the smooth sub-convergence of the evolution equation's solutions to critical points of the regularized energy in $W^{2,p}$. We then show that the solutions to the regularized evolution equations converge to a weak solution of the negative gradient flow of the $p$-elastic energies. These latter weak solutions also sub-converge to critical points of the $p$-elastic energy.
\end{abstract}

\maketitle
\tableofcontents

\section{Introduction}

We continue our investigation that was started in \cite{Blatt2021} of the negative gradient flow of the sum of $p$-elastic energy and a positive multiple $\lambda>0$ of the length of a regular curve $\g: \RZ \rightarrow \mathbb R^n$ , i.e. of 
$$
 \E (\g)= \tfrac 1 p \int_{\RZ} |\ka_{\g}|^p ds + \lambda \Le(\g).
$$
We will only consider the non-singular case  $p \in [2,\infty)$ of the energy $\EP$. Here, $\partial_s = \frac {\partial_x} {|\g'(x)|}$ denotes the derivative with respect to arc-length, $\ka_\g = \partial_s ^2 \g$  the curvature vector of the curve $\g$, $ds = |\g'| dx$  the integration with respect to arc-length, and $\Le(\g)= \int_{\RZ} |\g'| dx$ denotes the length of the curve.

In \cite{Blatt2021} we used de Giorgi's minimizing movement scheme together with approximate normal graphs to prove short time existence for the weak negative gradient flow of the energy $\E$. We could start the gradient flow for any initial regular curve of class $W^{2,p}$ and give a lower bound on the time of existence that only depended on $p,\lambda$ and the energy of the initial curve. 
 
 In this article we prove the following fundamental result.

\begin{theorem}\label{thm:LongTimeExistenceRElastic} \footnote{While proofreading a manuscript with a similar result for planar curves but a completely different approach appeared (c.f. \cite{Okabe2021}).}
Given any regular closed curve  $\g_0: \RZ \rightarrow \mathbb R^n$ parametrized with constant speed there is a family of regular curves $\g: [0,\infty) \times \RZ \rightarrow \mathbb R^n$, $\g \in H^1([0,\infty), L^2(\RZ, \R^n)) \cap L^\infty([0,\infty), W^{2,p}(\RZ, \R^n)) \cap C^{\frac 12} ([0,\infty) L^2 (\RZ, \mathbb R^n ))$ solving the initial value problem 
\begin{equation*}
	\begin{cases}
	\partial^\bot_t \g & = - \nabla_{L^2} \EP(\g)  \quad \text{ on } [0,T) \times \RZ \\
	\g(0, \cdot ) &= \g_0
	\end{cases}
\end{equation*}
in the weak sense, i.e. for all $V\in C^\infty_c ([0, \infty) \times \RZ, \R^n)$ we have 
\begin{equation*}
 \int_0^\infty \int_{\R/\Z} \langle \partial_t^\bot \g, V \rangle ds dt = - \int_0^\infty (\delta_V \EP (\g)) dt.
\end{equation*}
Furthermore, for all times $t \in [0,\infty)$ the curves $\g_t = \g(t, \cdot)$ are parametrized with constant speed and 
there is a subsequence $t_n \rightarrow \infty$ such that the translated curves
$
 f_{t_n }  (\cdot)- f_{t_n}(0)
$
converge in $W^{2,p}$ to a critical point of $\EP$.
\end{theorem}

Let us note that this result dramatically improves our previous findings in \cite{Blatt2021} and is based on a completely different approach. Instead of using de Giorgi's minimizing movement scheme, we approximate the solution by solutions  to the negative gradient flow of regularized energies. We cut off the degeneracy of the energy $\E^{(p)}$  near $\kappa =0$ and add a small multiple of higher order term to get an energy, for which one can prove long-time existence of the related gradient flow adapting the techniques in \cite{Dziuk2002}. In Section \ref{sec:APriori} we prove a priori estimates that we use in Section \ref{sec:Convergence} to show that the solutions of the gradient flow of the regularized energies converge to a solution of our initial problem.

Elastica have led to major breakthroughs ever since James Bernoulli had challenged the mathematical world to invent a mathematical model for the elastic beam and solve the resulting equations in 1691. Apart from the invention of the curvature of a curve by James Bernoulli, this led to the remarkable classification of planar elastica by Leonhard Euler  \cite{Euler1952, Oldfather1933} in Euclidean space and by Langer and Singer on the sphere and in hyperbolic space \cite{Langer1984}.  For a very nice introduction to the history of this problem we recommend the article \cite{Truesdell1983}.

In the last decades one has successfully investigated the related gradient flows in the quadratic case $p=2$. Both, for the flow in Euclidean space \cite{Dziuk2002} as well as for the flow of curves on manifolds \cite{DallAcqua2018,Mueller2019,Pozzetta2020}, one has now quite a complete picture of the behaviour or the evolution equation in this special case.

While the analysis of the Euler elastica goes back more than three centuries, the non-quadratic case $p \not= 2$ was treated only quite recently. Watanabe \cite{Watanabe2014} found that critical points of the $p$-elastic energy for $p \not=2$ can have a very different behavior than in the quadratic case .  He was able to classify all $p$-elastic curves using new variants of elliptic integrals.

The flow considered in this article was not investigated up to now apart from the work of the authors. In \cite{Okabe2018, Novaga2020} a second order version of the flow was considered for both curves and networks.

\section{The regularized equations}
 
To construct solutions to \eqref{regflow}, we approximate the energy $\EP$ by regularized energies. We get rid of the degeneracy of $\EP$ by introducing the regularization
$$
\EPD(\g) = \frac 1p \int_{\RZ} (|\ka|^2 + \delta^2 )^{\frac p2} ds
$$
for any $\delta >0 $. We will furthermore add a small positive multiple of the energy
$$
 \F(\g) = \frac 1 2 \int_{\RZ} |\nabla_s \ka|^2 ds,
$$
where $\nabla_s\phi = \partial_s \phi - \langle \partial_s \phi, \tau \rangle \tau$ indicates the normal part of $\partial_s \phi$. For any $p\geq2$, $\varepsilon >0$, $\delta > 0 $, and $\lambda > 0$ let
\begin{align}\label{energyreg}
	\ER(\g) = \varepsilon \F (\g) + \EPD (\g)+ \lambda \Ell(\g).
\end{align}
The negative $L^2$-gradient flow for the regularized energy $\ER$ is given by 
\begin{align}\label{regflow}
	\partial_t \g = - \nabla_{L^2}\ER(\g)  
\end{align}
where $ \nabla_{L^2}\E^p_{\varepsilon,\delta,\lambda}(f) =\varepsilon \nabla_{L^2}  \F (\g) +  \nabla_{L^2 } \EPD (\g) + \lambda \nabla_{L^2} \Ell(\g)$.
A straightforward calculation yields following \cite{Dziuk2002}
\begin{align}
\nabla_{L^2} \F(\g) &= - \left(\nabla_s^4 \kappa + |\kappa|^2 \nabla_s^2 \kappa + \langle \nabla_s \kappa, \kappa \rangle \nabla_s \kappa  -\tfrac 32 |\nabla_s \kappa |^2 \kappa\right) ,  \label{eq:GradientF1}\\
	\nabla_{L^2} \EPD (\g) & = (\kappa^2 + \delta^2)^{\frac{p-2}{2}} \left(\nabla_s^2 \kappa + |\kappa|^2 \kappa \right) \label{eq:GradientE1}
	\\ 
	&  \quad + (p-2) (\kappa^2 + \delta^2)^{\frac{p-4}{2}} \left( \langle \kappa, \nabla_s^2 \kappa \rangle \kappa + |\nabla_s \kappa|^2 \kappa  + 2 \langle \kappa, \nabla_s \kappa \rangle \nabla_s \kappa \right) \nonumber\\
	&\quad +  (p-4) (p-2) (\kappa^2 + \delta^2)^{\frac{p-6}{2}} \langle \kappa, \nabla_s \kappa\rangle^2 \kappa \nonumber \\
		&  \quad - \tfrac 1p(\kappa^2 + \delta^2)^{\frac{p}{2}} \kappa \nonumber 
\end{align}
and $\nabla_{L^2} \Le(\g) = - \kappa$. (For the convenience of the reader we will give the details of this calculation in the next section.)

This section is devoted to the proof of the following theorem. The proof follows the path paved in \cite{Dziuk2002}.
\begin{theorem} \label{thm:LongTimeExistenceRE}
	For any fixed numbers $p>2$, $\delta > 0$, $\varepsilon>0$, $\lambda >0$, and any smooth initial closed curve $f_0$, there exists a smooth solution $\g_t:\R/\Z\rightarrow \R^n$, $0\leq t < \infty$, to the regularized $L^2$-gradient flow \eqref{regflow} of $\ER$. Furthermore, after reparametrization by arc-length and suitable translation, as $t_i\rightarrow \infty$ the curves $f_{t_i}$ subconverge to a critical point of $\ER$.
\end{theorem}

\begin{remark}

\begin{enumerate}
\item Note that the asymptotic behavior stated in this theorem will not be used in the rest of the article. Indeed our proof of asymptotic behavior stated in Theorem \ref{thm:LongTimeExistenceRElastic} does not use the approximations at all.
\item One can prove full convergence of the flow as $t \rightarrow \infty$ using  {\L}ojasiewicz-Simon gradient estimates as in \cite{Chill2009, Pozzetta2020, Mantegazza2021 }. In contrast to that, full convergence for solutions of the negative gradient flow of the energy $\EP$ for $t \rightarrow \infty$ is a completely open problem.
\end{enumerate}
\end{remark}

\subsection{Equations of evolution and inequalities}

We will use the following toolbox from \cite{Dziuk2002} to derive the evolution equations of geometric quantities. 

 \begin{lemma}[{\cite[Lemma~2.1]{Dziuk2002}}]
 	Let $f: [0,T) \times \R/\Z\rightarrow \R^n$ be a time-dependent curve and $\phi$ any normal field along $f$. If $f$ satisfies $\partial_t f= V+\varphi \tau$, where $V$ is the normal velocity and $\varphi = \langle \partial_t f,\tau \rangle$, then we have
 	\begin{align}
 		\nabla_s \phi & = \partial_s \phi + \langle \phi,\kappa \rangle \tau, \\
 		\partial_t (ds) & = (\partial_s \varphi - \langle \kappa, V \rangle ) ds, \label{eq:eveqds} \\
 		\partial_t\partial_s -\partial_s  \partial_t & = (\langle \kappa, V \rangle - \partial_s \varphi) \partial_s, \\
 		\partial_t \tau &= \nabla_s V + \varphi \cdot \kappa ,\\
 		\partial_t \phi & = \nabla_t \phi - \langle \nabla_s V + \varphi k, \phi \rangle \tau, \label{eq:eveqphi} \\
 		\nabla_t \kappa &= \nabla_s^2 V + \langle \kappa, V \rangle \kappa + \varphi \cdot \nabla_s \kappa ,\label{eq:eveqkappa}\\
 		(\nabla_t\nabla_s - \nabla_s \nabla_t) \phi & = (\langle \kappa, V \rangle - \partial_s \varphi) \nabla_s \phi + \langle \kappa, \phi \rangle \nabla_s V - \langle \nabla_s V, \phi\rangle \cdot \kappa .\label{eq:eveqts}
 	\end{align}
 \end{lemma}
 
 Using the formulas above, a straightforward calculation yields the formulas for the $L^2$ gradients of $\F$ and $\EPD$ given in \eqref{eq:GradientF1} and \eqref{eq:GradientE1}. Equations \eqref{eq:eveqds} and \eqref{eq:eveqkappa} tell us that
 \begin{align*}
  \partial_t \EPD(\g_t) 
  &  = \int_{\RZ} (\kappa^2 + \delta ^2 )^{\frac {p-2} 2 } \langle \kappa , \partial_t \kappa \rangle ds + \frac 1p \int_{\RZ} ( \kappa ^2  + \delta ^2) ^{\frac p2} \partial_t (ds) \\
  & = \int_{\RZ} (\kappa ^2 + \delta ^2 ) ^{\frac {p-2} 2 } \langle \kappa, \nabla_s ^2 V \rangle ds+  \int_{\RZ}  (\kappa ^2 + \delta ^2 ) ^{\frac {p-2} 2 } |\kappa|^2 \langle \kappa, V \rangle  ds \\
  & \quad - \frac 1p   \int_{\RZ} ( \kappa^2 + \delta^2 ) ^{\frac p2} \langle \kappa, V \rangle ds \\
  & = \int_{\RZ} \langle \nabla_s^2 ( (\kappa ^2 + \delta ^2 ) ^{\frac {p-2} 2 } \kappa),V \rangle ds+  \int_{\RZ}  (\kappa ^2 + \delta ^2 ) ^{\frac {p-2} 2 } |\kappa|^2 \langle \kappa, V \rangle  ds \\
  & \quad - \frac 1p  \int_{\RZ} ( \kappa^2 + \delta^2 ) ^{\frac p2} \langle \kappa, V \rangle ds,
\end{align*}
where we used integration by parts in the last step. So we get
$$
 \nabla_{L^2} \EPD (\g) = \nabla_s ^2 ((\kappa^2 + \delta ^2 )^{\frac {p-2} 2 }  \kappa ) + (\kappa^2 + \delta^2 )^{\frac {p-2} 2} |\kappa|^2 \kappa - \frac 1p (\kappa^2 + \delta ^2 )^{\frac p2} \kappa. 
$$
Together with 
\begin{align*}
 \nabla_s ^2 ((\kappa^2 + \delta ^2 )^{\frac {p-2} 2 }  \kappa ) & = (\kappa^2 + \delta ^2 )^{\frac {p-2} 2 }\nabla^2_s \kappa + 2 (p-2) (\kappa^2 + \delta^2)^{\frac {p-4} 2 } \langle \kappa, \nabla_s \kappa \rangle \nabla_s \kappa 
 \\ & \quad +  (p-2)  \big((\kappa^2 + \delta^2 )^{\frac {p-4} 2 } \langle \kappa, \nabla^2_s \kappa \rangle 
 +  (\kappa^2 + \delta ^2)^{\frac {p-4} 2 } |\nabla_s \kappa|^2 
 \\ & \qquad+ (p-4) (\kappa^2 + \delta^2) ^{\frac {p-6} 2} \langle \kappa, \nabla_s \kappa \rangle^2 \big) \kappa
\end{align*}
this proves \eqref{eq:GradientE1}.

To calculate the $L^2$ gradient of $\F$ we first observe using \eqref{eq:eveqds}
\begin{align*}
 \partial_t \F (\g) & = \int_{\RZ} \langle \nabla_s \kappa, \nabla_t \nabla_s \kappa\rangle ds + \frac 1 2 \int_{\RZ} |\nabla_s \kappa|^2 \partial_t (ds) 
 \\ &  =  \int_{\RZ} \langle \nabla_s \kappa, \nabla_t \nabla_s \kappa\rangle ds - \frac 1 2 \int_{\RZ} |\nabla_s \kappa|^2 \langle \kappa, V \rangle ds
\end{align*}
From \eqref{eq:eveqts} and \eqref{eq:eveqkappa} we get
\begin{align*}
 \nabla_t \nabla_s \kappa &= \nabla_s (\nabla^2_s V + \langle \kappa, V \rangle \kappa) +  \langle \kappa, V \rangle \nabla_s \kappa + |\kappa|^2 \nabla_s V - \langle \nabla_s V , \kappa \rangle \kappa
 \\ & = \nabla^3 _s V + \langle \nabla_s \kappa, V \rangle \kappa + 2 \langle \kappa , V \rangle \nabla_s \kappa + |\kappa|^2 \nabla_s V
\end{align*}
and hence 
\begin{align*}
 \partial_t \F (\g)  &  =  \int_{\RZ} \langle \nabla_s \kappa,\nabla^3 _s V \rangle ds + \int_{\RZ} |\kappa|^2 \langle \nabla_s \kappa, \nabla_s V \rangle ds  
 \\ & \quad + \int_{\RZ} \langle(\langle \nabla \kappa, \kappa \rangle \nabla \kappa + 2|\nabla_s \kappa|^2 \kappa ), V \rangle ds - \frac 1 2 \int_{\RZ} |\nabla_s \kappa|^2 \langle \kappa, V \rangle ds \\
 & =  - \int_{\RZ} \langle \nabla^4_s \kappa,V \rangle ds - \int_{\RZ} \langle \nabla_s (|\kappa|^2  \nabla_s \kappa), V \rangle ds  
 \\ & \quad + \int_{\RZ} \langle(\langle \nabla \kappa, \kappa \rangle \kappa + \frac 32|\nabla_s \kappa|^2 \kappa ), V \rangle ds,
\end{align*}
where we used integration by parts in the last step. From this we can read off that
\begin{align*}
 \nabla_{L^2} \F (\g) &= - \nabla_s ^4 \kappa - \nabla_s (|\kappa|^2 \nabla_s \kappa) + \langle \nabla_s \kappa, \kappa \rangle \nabla_s \kappa + \frac 32 |\nabla_s \kappa|^2 \kappa 
 \\ & = - \nabla_s ^4 \kappa - |\kappa|^2 \nabla^2_s \kappa -  \langle \kappa, \nabla_s \kappa \rangle \nabla_s \kappa  + \frac 32 |\nabla_s \kappa|^2 \kappa 
\end{align*}
Thus Equation \eqref{eq:GradientF1} is proven. 


As often the precise algebraic form of the terms does not matter, we will use the $\ast$ notation introduce in \cite{Dziuk2002} to shorten the notation.
For vectors $\phi_1, \ldots, \phi_k$ along $\g$ we denote by $\phi_1 \ast \cdots \ast \phi_k$ any multilinear combination of these vectors.
 For $p>2$ and $\delta >0$ we let $\psi : \R\rightarrow \R$ be given by $\psi (x) = (x^2 + \delta^2)^{\frac p2}$.
For a vector field $\phi$, we denote by $P^{\mu}_{\nu}(\phi)$ any linear combination of terms of the type 
$$
(\psi^{(i_0)}(|\kappa|) )^j\gs^{i_1}\phi\ast\cdots\ast\gs^{i_\nu}\phi
$$
 with universal constant coefficients where $\mu$ is greater than or equal to the total number of derivatives $i_1+\cdots+i_\nu$, $i_l \in\N_0$, and $j\in \{0,1\}$. 
If $P^\mu_{\nu}(\phi)$ only contains terms with derivatives of $\phi$ up to order $\omega\in\N$, we indicate that by writing $P^{\mu,\omega}_{\nu} (\phi)$ if needed. We will write $P^{\mu}(\phi)$ for any linear combination of terms $P^{\mu}_\nu$, $\nu \in \mathbb N$, and $P^{\mu, \omega}(\phi)$ for any linear combination of terms $P^{\mu, \omega}_{\nu}$, $\nu \in \mathbb N$. 

Note that the factor $(\psi^{(i_0)}(\kappa) )^j$ hardly plays a role in the upcoming computations as along the flow $\F$ is bounded by the initial energy and hence by Sobolev's embeddings $\|\kappa\|_{L^\infty}$ is bounded as well. Hence, the factor is controlled by a constant dependent on $\delta$. Notice also that the formula $\gs P^\mu(\kappa) = P^{\mu+1}(\kappa)$ holds and that $P^\mu(\kappa)$ can contain summands  $P^{\tilde{\mu}} (\kappa)$ of lower order $0 \leq \tilde{\mu} \leq \mu-1$.

Using this notation, the evolution equation \eqref{regflow} reduces to 
$$
 \partial_t f = \varepsilon(\gs^4\kappa + P^2(\kappa)) + \lambda \kappa + P^2(\kappa).
$$

We will use the following lemma to capture the main structure of the evolution equation of the $L^2$ norm of higher derivatives of $\kappa$.
\begin{lemma}\label{lem:eveq1}
	Suppose $f:[0,T) \times \R/\Z\rightarrow \R^n$ moves in a normal direction with velocity $\ft f = V$, $\phi$ is a normal vector field along $f$, $\delta >0$, and $\gt\phi -\varepsilon \gs^6 \phi = Y$. Then 
	\begin{align}\label{eq:eveqintphi}
		\frac{d}{dt} \frac 12 \int_{\R/\Z} |\phi|^2 ds + \varepsilon \int_{\R/\Z} |\gs^3\phi|^2 ds = \int_{\R/\Z} \langle Y,\phi\rangle ds - \frac 12 \int_{\R/\Z} |\phi|^2 \langle\kappa,V\rangle ds.
	\end{align}
	Furthermore $\chi = \gs\phi$ satisfies the equation
	\begin{align}\label{eq:eveqchi}
		\gt \chi - \varepsilon \gs^6\chi= \gs Y + \langle \phi,\kappa\rangle\gs V - \langle \phi, \gs V\rangle \kappa + \langle \kappa, V\rangle \chi
	\end{align}
\end{lemma}
\begin{proof} Analogous to \cite[Lemma~2.2]{Dziuk2002}, \eqref{eq:eveqintphi} follows from the evolution equations \eqref{eq:eveqds} and \eqref{eq:eveqphi} as well as \eqref{eq:eveqchi} from \eqref{eq:eveqts}. 
\end{proof}

We observe as in {\cite[Lemma~2.3]{Dziuk2002}}:
\begin{lemma}\label{lem:eveq2}
	Suppose $\partial_t f = \varepsilon(\gs^4\kappa + P^2(\kappa)) + \lambda \kappa + P^2(\kappa)$, where $\varepsilon,\lambda>0$. Then for any $m\geq 0$ the derivatives of the curvature $\phi_m = \gs^m\kappa$ satisfy
	\begin{align*}
		\gt \phi_m - \varepsilon \gs^6\phi_m = \varepsilon P^{m+4}(\kappa) + \lambda (\gs^{m+2}\kappa + P^m(\kappa)) + P^{m+4}(\kappa).
	\end{align*}
\end{lemma}
\begin{proof}
	The case of $m=0$ is a consequence of \eqref{eq:eveqkappa}. By applying \eqref{eq:eveqchi} we obtain inductively
	\begin{align*}
		&\gt \phi_{m+1} - \varepsilon \gs^6\phi_{m+1} \\
		&= \gs (\gt \phi_{m} - \varepsilon \gs^6\phi_{m}) + (\langle \phi_m,\kappa\rangle\gs \partial_t f - \langle \phi_m, \gs \partial_t f\rangle \kappa + \langle \kappa, \partial_t f\rangle \phi_{m+1}) \\
		& = \varepsilon P^{m+5}(\kappa) + \lambda (\gs^{m+3}\kappa + P^{m+1}(\kappa)) + P^{m+5}(\kappa) \\
		& \hspace{1em} + (\varepsilon P^{m+5}(\kappa) + \lambda P^{m+1}(\kappa) + P^{m+5}(\kappa)) \\
		& = \varepsilon P^{m+5}(\kappa) + \lambda (\gs^{m+3}\kappa + P^{m+1}(\kappa)) + P^{m+5}(\kappa),
	\end{align*}
	which proves the lemma.
\end{proof}

Counting the number of factors containing $\gs\kappa$ or higher derivatives thereof, we obtain the following fact about terms of type $P^{\mu, \omega}(\kappa).$
\begin{lemma}\label{lem:kappatogs}
	Let $f: \R/\Z \rightarrow \R^n$ be a smooth closed curve and $\mu\in\N$. Then we have
	\begin{align}\label{eq:Pkappa}
		P^{\mu,\omega}(\kappa) = \sum_{k=1}^{\mu}P^{\mu - k,\omega-1}_{k} (\gs \kappa) \ast P^0(\kappa) + P^0(\kappa),
	\end{align}
	for any $\omega \in \N$, $1 \leq \omega \leq \mu$.
\end{lemma}

\begin{proof}
A term of type $P^{\mu, \omega}(\kappa)$ consists of a linear combination of terms 
$$
(\psi^{(i_0)}(|\kappa|) )^j\gs^{i_1}\kappa\ast\cdots\ast\gs^{i_\nu}\kappa
$$
 with universal constant coefficients where $\mu$ is greater than or equal to the total number of derivatives $i_1+\cdots+i_\nu$, $i_l \in \{0, \ldots, \omega\}$, and $j\in \{0,1\}$. Changing the order of indices we can achieve that $i_1, \ldots, i_k \geq 1$ and $i_{k+1}, \ldots, i_\nu =0$. Since $i_1 + \cdots + i_k \leq \mu$, we obtain $k \leq \mu$. Hence,
\begin{align*}
 (\psi^{(i_0)}(|\kappa|) )^j\gs^{i_1}\kappa\ast\cdots\ast\gs^{i_\nu}\kappa & = ((\psi^{(i_0)}(|\kappa|) )^j\gs^{i_1}\kappa\ast\cdots\ast\gs^{i_k}\kappa) ( \gs^{i_{k+1}}\kappa\ast\cdots\ast\gs^{i_\nu} )
 \\ & = P^{\mu - k,\omega-1}_k (\nabla_s \kappa) \ast P^0 (\kappa).
\end{align*}
\end{proof}

In a next step, we state a variant of the Gagliardo-Nirenberg interpolation inequalities for higher order curvature functionals for curves in $\R^n$. For that we define scale invariant norms $\|\kappa\|_{k,q} = \sum_{i=0}^k \|\gs^i \kappa \|_q$, where
\begin{align}\label{def:scaleinvariantnorm}
	\|\gs^i \kappa \|_q = \Ell(f)^{i+1-\tfrac 1q} \left(\int_{\R/\Z} |\gs^i \kappa |^q ds \right)^{\tfrac 1q}.
\end{align}

\begin{lemma}[{\cite[Lemma~2.4]{Dziuk2002}}]\label{lem:interpolationkappa}
	For any smooth closed curve $f: \R/\Z \rightarrow \R^n$ and any $k\in\N$, $q\geq 2$, and $0 \leq i < k$ we have 
	\begin{align*}
		\| \gs^i \kappa \|_q \leq c \|\kappa\|_2^{1-\alpha} \|\kappa\|^\alpha_{k,2},
	\end{align*}
	where $\alpha = (i+\tfrac 12 - \tfrac 1q)/k$ and $c=c(n,k,q)$.
\end{lemma}

An immediate consequence of this lemma for terms of type $P^{\mu, k-1}_\nu$ is the following:
\begin{lemma}[{\cite[Proposition~2.5]{Dziuk2002}}]\label{lem:interpolationPkappa}
	Let $k\in\N$ and $f$ as in the previous lemma. For any term $P^{\mu,k-1}_\nu(\kappa)$ with $\nu \geq 2$, we have 
	\begin{align}\label{eq:interpolationP}
		\int_{\R/\Z} |P^{\mu,k-1}_\nu (\kappa)| ds \leq c \mathcal{L}(f)^{1-\mu-\nu} \|\kappa\|_{L^\infty} \|\kappa\|_2^{\nu-\gamma} \|\kappa\|^\gamma_{k,2},
	\end{align}
	where $\gamma = (\mu+ \tfrac 12 \nu -1)/k$ and $c= c(n,k,\mu,\nu, \delta)>0$. Moreover, if $\mu + \tfrac 12 \nu < 2k+1$, then $\gamma <2$ and we have for any $\eta > 0$ 
	\begin{align}\label{eq:interpolationeta}
		& \int_{\R/\Z} |P^{\mu,k-1}_\nu (\kappa)| ds \nonumber\\
		& \leq c \|\kappa\|_{L^\infty}\left( \eta \int_{\R/\Z} |\gs^k \kappa |^2 ds +  \eta^{-\frac{\gamma}{2-\gamma}}\left(\int_{\R/\Z} |\kappa|^2 ds\right)^{\frac{\nu-\gamma}{2-\gamma}} +  \mathcal{L}(f)^{1-\mu-\frac\nu 2} \left(\int_{\R/\Z} |\kappa|^2 ds \right)^{\frac \nu 2} \right)
	\end{align}
	for another constant $c= c(n,k,\mu,\nu, \delta)>0$.
\end{lemma}

\begin{proof} 
	We prove this statement along the lines of \cite[Proposition~2.5]{Dziuk2002}.
 	First note that $\kappa$ is uniformly bounded on $\R/\Z$, it is therefore sufficient to focus on $\nabla_s^{i_1} \kappa \ast \cdots \ast \nabla_s^{i_\nu} \kappa$ with $i_1 + \cdots + i_\nu = \mu$ and $i_l \leq k-1$ instead of $P_\nu^{\mu,k-1}(k)$ subsequently. We then observe by H\"older's inequality, the notion of scale invariant norms in \eqref{def:scaleinvariantnorm}, and Lemma~\ref{lem:interpolationkappa} that
 	\begin{align*}
 		\int_{\R/\Z} |\nabla_s^{i_1} \kappa \ast \cdots \ast \nabla_s^{i_\nu} \kappa| ds &\leq \prod_{l=1}^\nu \|\nabla_s^{i_l} \kappa \|_{L^\nu} 
 		= \mathcal{L}^{1-\mu-\nu}(f) \prod_{l=1}^\nu \|\nabla_s^{i_l} \kappa \|_{\nu}  \\
 		& \leq c  \mathcal{L}^{1-\mu-\nu}(f) \prod_{l=1}^\nu \|\kappa\|_2^{1-\alpha_l} \|\kappa\|^{\alpha_l}_{k,2}
	\end{align*}
	for $\alpha_l = (i_l + \tfrac 12 - \tfrac 1\nu)/k$ and some positive constant $c=c(n,k,\nu)$. As $\sum_{l=1}^\nu \alpha_l = \gamma$, the first inequality \eqref{eq:interpolationP} follows. For the second claim \eqref{eq:interpolationeta} we recall the following standard interpolation result
	\begin{align}\label{eq:standInterpolation}
		\|\kappa\|_{k,2}^2 \leq c(k) (\|\nabla_s^k\kappa \|_2^2 + \|\kappa\|_2^2),
	\end{align}
	from which we deduce together with $\gamma<2$ and the equivalence of $p$-norms on $\R^2$ 
	\begin{align*} 
		\|\kappa\|_2^{\nu-\gamma} \|\kappa\|^\gamma_{k,2} \leq c(k) (\|\kappa\|_2^{\nu-\gamma} \|\nabla_s^k\kappa \|_2^\gamma + \|\kappa\|_2^\nu).
	\end{align*}
	Now taking account of the scaling and applying Young's inequality for $\eta>0$, we observe 
	\begin{align}\label{eq:interpolSecStep}
		\mathcal{L}(f)^{1-\mu-\nu} \|\kappa\|_2^{\nu-\gamma} \|\kappa\|^\gamma_{k,2} &  \leq c \mathcal{L}(f)^{1-\mu-\nu} (\|\kappa\|_2^{\nu-\gamma} \|\nabla_s^k\kappa \|_2^\gamma + \|\kappa\|_2^\nu) \nonumber \\
		& \leq  c \|\kappa\|_{L^2}^{\nu-\gamma} \|\nabla_s^k\kappa \|_{L^2}^\gamma + c \mathcal{L}(f)^{1-\mu-\frac\nu 2} \|\kappa\|_{L^2}^\nu \nonumber \\
		& \leq \eta \|\nabla_s^k\kappa \|_{L^2}^\gamma + c \eta^{-\frac{\gamma}{2-\gamma}} \|\kappa\|_{L^2}^{2\frac{\nu-\gamma}{2-\gamma}} + c \mathcal{L}(f)^{1-\mu-\frac\nu 2} \|\kappa\|_{L^2}^\nu 
	\end{align}
	for some constant $c = c(k)>0$. 

\end{proof}

The interpolation inequality \eqref{eq:interpolationeta} can be transfered to any term $P^\mu_\nu(\gs\kappa )$ that only involves derivatives of $\gs\kappa$. The aim is to interpolate the $L^1$-norm of $P^\mu_\nu(\gs\kappa )$ between the $L^2$-norms of $\gs \kappa$ and $\gs^{k+1}\kappa$.

\begin{lemma}\label{lem:interpolationgskappa}
	Let $k\in\N$ and $f$ as in Lemma~\ref{lem:interpolationkappa}. For any term $P^{\mu,k-1}_\nu(\gs\kappa )$ with $\nu \geq 2$ such that $\mu + \tfrac 12 \nu < 2k+1$ and $\eta > 0$, we get 
	\begin{align}
	& \int_{\R/\Z} |P^{\mu,k-1}_\nu (\gs \kappa)| ds \nonumber\\
	& \leq c \|\kappa\|_{L^\infty} \Bigg(\eta \int_{\R/\Z} |\gs^{k+1} \kappa |^2 ds +  \eta^{-\frac{\gamma}{2-\gamma}}\left(\int_{\R/\Z} |\gs\kappa|^2 ds\right)^{\frac{\nu-\gamma}{2-\gamma}} 
	\\ &\quad +  \mathcal{L}(f)^{1-\mu-\frac\nu 2} \left(\int_{\R/\Z} |\gs\kappa|^2 ds \right)^{\frac \nu 2} \Bigg),
	\end{align}
	where $\gamma = (\mu+ \tfrac 12 \nu -1)/k$ and $c=c(n,k,\mu,\nu, \delta)>0$.
\end{lemma}
\begin{proof}
	Lemma~\ref{lem:interpolationkappa} also applies to derivatives of $\gs\kappa$, more precisely we have
	\begin{align*}
		\| \gs^i (\gs\kappa) \|_q \leq c \|\gs \kappa\|_2^{1-\alpha} \|\gs\kappa\|^\alpha_{k,2},
	\end{align*}
	where $\alpha = (i+\tfrac 12 - \tfrac 1q)/k$ and $c=c(n,k,q)$.
	Therefore, the argument in the proof of Lemma \ref{lem:interpolationPkappa} yields the statement.
\end{proof}

The next two lemmata illustrate some relations between the full derivatives and normal derivatives of $\kappa$. 
Here $Q^\mu_\nu(\kappa)$ denotes any linear combination of terms of the type $\gs^{i_1} \kappa \ast\cdots \ast \gs^{i_\nu} \kappa $, where $i_1 + \cdots + i_\nu = \mu$.

\begin{lemma}[{\cite[Lemma~2.6]{Dziuk2002}}]\label{lem:gsidentities}
	We have the identities
	\begin{align*}
		\gs\kappa - \partial_s\kappa &= |\kappa|^2\tau, \\
		\gs^m\kappa - \fs^m\kappa &= \sum_{i=1}^{\left[\frac m2\right]} Q^{m-2i}_{2i+1}(\kappa) + \sum_{i=1}^{\left[\frac{m+1}{2}\right]} Q^{m+1-2i}_{2i}(\kappa) \tau.
	\end{align*}
\end{lemma}

One can inductively derive the following estimates using these equalities. They allow us to derive estimates for the full derivatives from estimates of the normal derivatives.

\begin{lemma}[{\cite[Lemma~2.7]{Dziuk2002}}]\label{lem:uniformbounds}
	Assume the bounds $\|\kappa\|_{L^2} \leq \Lambda_0$ and $\|\gs^m\kappa\|_{L^1} \leq \Lambda_m$ for $m\geq 1$. Then for any $m\geq 1$ one has 
	\begin{align*}
		\|\fs^{m-1}\kappa\|_{L^\infty} + \|\fs^m\kappa\|_{L^1} \leq c_m(\Lambda_0,\ldots,\Lambda_m).
	\end{align*}
\end{lemma}

\subsection{Short-time existence}
For the flow considered in this paper, short-time existence is a standard matter. Therefore, we only shortly sketch  the prove for smooth initial data. 

Using $\nabla_s \phi = \partial_s \phi + \langle \kappa, \phi \rangle \tau$, we obtain 
$$
 \nabla_s^4 \kappa = |\g'|^{-6} (\partial_x^6 \g - \langle \partial_x^6 \g, \tau \rangle  \tau ) +  \text{ lower order terms}.
$$
So using standard theory, we get a smooth solution $\tg: \R/\Z \times [0, T)\rightarrow \R^n$ to the quasilinear parabolic equation
$$
 \partial_t \tg = \nabla_{L^2} \ER(\tg) + |\tg'|^{-6} \langle \partial_x^6 \tg, \tilde \tau \rangle \tilde \tau
$$
with initial data $\g_0$, i.e. a smooth family of smooth curves with the right normal velocity. We set $ \sigma(t, \cdot) =  |\tg_t'|^{-7} \langle \partial_x^6 \tg_t, \tilde \tau_t \rangle$ and solve the initial value problem $\partial_t \phi = \sigma (t, \phi)$ with $\phi(0, \cdot) = id$. Then one easily calculates that $\g(t,x) = \g(t, \phi(t,x))$ solves the initial value problem \eqref{regflow}. 
\subsection{Long-time existence and asymptotic behavior}

As mentioned before, we can use the $\ast$-factor notation and its linear combinations to write the evolution equation \eqref{regflow} in the form
\begin{align}\label{simpflow}
		\partial_t f = \varepsilon(\gs^4 \kappa + P^2(\kappa)) + P^2(\kappa) + \lambda \kappa.
	\end{align}
	
Let us first deduce some obvious uniform bounds on $\mathcal{L}(f_t)$, $\|\kappa \|_{L^\infty}$, and $\|\gs\kappa\|_{L^2}$ for solutions of \eqref{regflow}.
Note that 	$\E^p_{\varepsilon,\delta,\lambda} (f_t)$ is monotonically decreasing in time $t$  as
\begin{equation} \label{eq:energyidentity}
  \frac d {dt} \ER(\g_t) = - \int_{\R/\Z} |\nabla_{L^2} \ER(\g_t)|^2 ds = - \int_{\R/\Z} | \partial_t \g_t|^2 ds,
\end{equation}
 we therefore obtain from the non-negativity of each term in the energy \eqref{energyreg}
\begin{align}\label{eq:proof2}
		 \lambda \mathcal{L}(f_t) &  \leq \E^p_{\varepsilon,\delta,\lambda}(f_t)  \leq \E^p_{\varepsilon,\delta,\lambda}(f_0), \nonumber \\
		 \int_{\R/\Z} |\gs\kappa|^2 ds & \leq \tfrac 2\varepsilon \E^p_{\varepsilon,\delta,\lambda}(f_0), \nonumber\\
		\int_{\R/\Z} |\kappa|^p ds & \leq  \int_{\R/\Z} (\kappa^2+\delta^2)^{\frac p2} ds \leq p \, \E^p_{\varepsilon,\delta,\lambda}(f_0)
	\end{align}
	for all $T>t>0$.	
	Fenchel's and H\"older's inequality imply that
	$$
		2\pi \leq \int_{\R/\Z} |\kappa| ds \leq \Ell(\g)^{1-\frac 1p}\left(\int_{\R/\Z} |\kappa|^p ds\right)^{\frac 1p}
	$$
	and hence
	\begin{align}\label{eq:proof6}
		 \Ell(\g_t) \geq  \frac {(2\pi)^{\frac p {p-1} }} { \left(\int_{\R/\Z} |\kappa|^p ds\right)^{\frac 1 {p-1}}}  \geq  \frac {(2\pi)^{\frac p {p-1} }} { (p\,  \ER(\g_0))^{\frac 1 {p-1}}} 
	\end{align}
	From the inequalities in \eqref{eq:proof2} and \eqref{eq:proof6} we thereby gain the following bounds for both the length und the $L^2$ norm of $\nabla_s \kappa$: 
	\begin{align}
		\frac {(2\pi)^{\frac p {p-1} }} { (p\, \ER(\g_0))^{\frac 1 {p-1}}}  \leq \mathcal{L}(f_t) &\leq \frac{\E^p_{\varepsilon,\delta,\lambda}(f_0)}{\lambda},\label{eq:uniformboundL} \\
		  \|\gs\kappa \|_{L^2}^2  &\leq \tfrac 2{\varepsilon} \E^p_{\varepsilon,\delta,\lambda}(f_0) .\label{eq:uniformboundgskappa}
	\end{align}
	Also the $L^\infty$ norm of $\kappa$ is bounded, which can be seen as follows. H\"older's inequality gives
\begin{align*}
		\|\kappa\|_{L^2} & \leq \mathcal{L}(f_t)^{\frac 12 - \frac 1p}\|\kappa\|_{L^p} \leq \mathcal{L}(f_t)^{\frac 12 - \frac 1p} (\int_{\R/\Z} |\kappa^2 + \delta^2|^{\frac p2} ds)^{\frac 1p} \leq ( p\lambda)^{\frac 1p} \lambda^{-\frac 12}  \E^p_{\varepsilon,\delta,\lambda}(f_0)^{\frac 12}, 
\end{align*}
and
\begin{align*}
		\|\gs\kappa\|_{L^1} & \leq \mathcal{L}(f_t)^{\frac 12} \|\gs\kappa\|_{L^2} \leq  (\tfrac{\lambda \varepsilon}{2})^{-\frac 12} \E^p_{\varepsilon,\delta,\lambda}(f_0).
	\end{align*}
Hence, Lemma \ref{lem:uniformbounds} tells us that 
	\begin{align}\label{eq:uniformboundkappa}
		\|\kappa\|_{L^\infty} < c(p,\varepsilon,\lambda, \E^p_{\varepsilon,\delta,\lambda}(f_0)).
	\end{align}

	To establish the long-time existence of a solution, let us assume that the maximal time of smooth existence for the flow \eqref{regflow}  $T$ satisfies $T < \infty$, i.e. let us assume that the flow does not exist for all times. Considering Lemma \ref{lem:eveq1}, Lemma \ref{lem:eveq2}, and the flow equation \eqref{simpflow}, we get
	\begin{equation} \label{eq:proof1}
	\begin{aligned}
		 & \frac{d}{dt} \frac 12 \int_{\R/\Z} |\gs^m\kappa|^2 ds  + \varepsilon \int_{\R/\Z} |\gs^{m+3}\kappa|^2 ds \\
		& \quad = \varepsilon \int_{\R/\Z} \langle P^{m+4}(\kappa) - \tfrac 12 \langle \gs^4 \kappa + P^2(\kappa),\kappa \rangle \gs^m \kappa , \gs^m \kappa\rangle ds  \\
		& \qquad+ \int_{\R/\Z} \langle P^{m+4}(\kappa) -\tfrac 12 \langle P^2(\kappa),\kappa \rangle \gs^m \kappa , \gs^m \kappa\rangle ds  \\
		& \qquad+ \lambda \int_{\R/\Z} \langle \gs^{m+2} \kappa + P^m(\kappa) -\tfrac 12 |\kappa|^2 \gs^m \kappa, \gs^m \kappa \rangle ds  \\
		& \quad = \epsilon \int_{\R/\Z} \langle P^{m+4}(\kappa),\gs^m \kappa\rangle ds +\int_{\R/\Z} \langle P^{m+4}(\kappa),\gs^m \kappa\rangle ds 
		\\ & \qquad 
		+ \lambda \int_{\R/\Z} \langle P^{m+2}(\kappa),\gs^m \kappa\rangle ds .
	\end{aligned}
	\end{equation}
	In order to apply interpolation estimates on the right-hand side, we need to integrate by parts once or twice, in case that terms like $\gs^{m+3}\kappa$ or $\gs^{m+4}\kappa$ appear in $P^{m+4}(\kappa)$, to achieve from \eqref{eq:proof1}
	\begin{align}\label{eq:proof3}
	& \frac{d}{dt} \frac 12 \int_{\R/\Z} |\gs^m\kappa|^2 ds  + \varepsilon \int_{\R/\Z} |\gs^{m+3}\kappa|^2 ds \nonumber \\
	& \hspace{2em}= \varepsilon \int_{\R/\Z} P^{2m+4,m+2}(\kappa) ds  +\int_{\R/\Z} P^{2m+4,m+2}(\kappa) ds + \lambda \int_{\R/\Z} P^{2m+2,m+2}(\kappa) ds .
	\end{align}
	It suffices to estimate, as $P^{2m+2,m+2}(\kappa)$ is included in the notation of $P^{2m+4,m+2}(\kappa)$,  by using Lemma \ref{lem:kappatogs} and Hölder's inequality
	\begin{align*}
		\int_{\R/\Z} &P^{2m+4,m+2}(\kappa) ds \\ & \leq \int_{\R/\Z} |P^{2m+4,m+2}(\kappa)| ds \nonumber \\
		& \leq \sum_{i=1}^{2m+4}
				\int_{\R/\Z} |P_{i}^{2m+4 - i,m+1}(\gs\kappa)| \cdot |P^0(\kappa)| ds + \int_{\R/\Z} |P^0(\kappa)| ds \nonumber \\
		& \leq \sum_{i=1}^{2m+4}
				c_0  \|\kappa\|_{L^\infty}  \int_{\R/\Z} |P_{i}^{2m+4 - i,m+1}(\gs\kappa)| ds + c_1 \|\kappa \|_{L^\infty}
	\end{align*}
	for some constants $c_0,c_1 > 0$ depending on $m$ and $\delta$. Note that the case $i=1$ can be ruled out thanks to the underlying precise gradient flow equation, including the gradients \eqref{eq:GradientF1} and \eqref{eq:GradientE1}, and integration by parts if needed. \\
	Now let es assume the initial bound $\E^p_{\varepsilon,\delta,\lambda}(f_t) \leq \Lambda$ for $0\leq t<T$. Then the previous inequality  reduces by the uniform bound on $\kappa$ in \eqref{eq:uniformboundkappa} to
	\begin{align}\label{eq:proof7}
	& \int_{\R/\Z} P^{2m+4,m+2}(\kappa) ds \nonumber \\
	& \leq \sum_{i=2}^{2m+4}	 c_0 (m,p,\varepsilon,\delta,\lambda,\Lambda)  \int_{\R/\Z} |P_{i}^{2m+4-i,m+1}(\gs\kappa)| ds + c_1 (m,p,\varepsilon,\delta,\lambda,\Lambda).
	\end{align}
	Finally, we interpolate the remaining integral between $\|\gs^{m+3} \kappa \|_{L^2}$ and $\|\gs \kappa\|_{L^2}$ by Lemma~\ref{lem:interpolationgskappa} with $k= m+2$, $\mu = 2m+4-i$, $\nu = i$, and $\gamma = \tfrac{2m+4-\frac i2 -1}{m+2}$. Note that Lemma~\ref{lem:interpolationgskappa} is applicable since $\kappa$ is uniformly bounded by \eqref{eq:uniformboundkappa}, $i=\nu\geq 2$, and $\mu + \frac 1 2 \nu = 2m+4 -i + \frac 12 i < 2m+4 +1 = 2k+1 $.
	So we have for any small $\eta > 0$ that
	\begin{equation}\label{eq:proof5intpol}
	\begin{aligned}
		& \int_{\R/\Z} |P_{i}^{2m+4-i,m-1}(\gs\kappa)| ds \\
		& \leq c_2 \Bigg( \eta \int_{\R/\Z} |\gs^{m+3}\kappa|^2 ds +  \eta^{-\frac{\gamma}{2-\gamma}} \left(\int_{\R/\Z} |\gs \kappa|^2 ds\right)^{\frac{i- \gamma}{2-\gamma}} \\ & \qquad +  \mathcal{L}(f)^{-2m-3+\frac i2}  \left(\int_{\R/\Z} |\gs \kappa|^2 ds\right)^{\frac i 2}\Bigg)
	\end{aligned}
	\end{equation}
	for a constant $c_2=c_2(i,m,n,p,\varepsilon,\delta,\lambda,\Lambda)>0$. By using $2\leq i \leq 2m+4$, the bound on $\|\gs\kappa\|_{L^2}$ in \eqref{eq:uniformboundgskappa}, and the boundedness of the length term in \eqref{eq:uniformboundL}, the interpolation estimate \eqref{eq:proof5intpol} reduces to 
	\begin{align}\label{eq:proof5intpolred}
	\int_{\R/\Z} |P_{i}^{2m+4-i,m+1}(\gs\kappa)| ds \leq \eta c_2 \int_{\R/\Z} |\gs^{m+3}\kappa|^2 ds + c (m,n,p,\varepsilon, \delta ,\lambda, \Lambda, \eta).
	\end{align}
	Hence, the estimates \eqref{eq:proof3}, \eqref{eq:proof7}, and \eqref{eq:proof5intpolred} 
	\begin{align}\label{eq:proof8}
		\frac{d}{dt} \int_{\R/\Z} |\gs^m\kappa|^2 ds  + \varepsilon \int_{\R/\Z} |\gs^{m+3}\kappa|^2 ds \leq c(m,n,p,\varepsilon,\delta,\lambda,\Lambda),
	\end{align}
	where we absorbed $\|\gs^{m+3}\kappa\|_{L^2}^2$, that appears on the right-hand side in \eqref{eq:proof5intpolred}, by choosing $\eta = \tfrac{\varepsilon}{2c_0c_2(2m+4)} > 0$ small enough. 
	So we get by estimate \eqref{eq:proof8}
	\begin{align*}
		\|\gs^m\kappa\|_{L^2}^2 (t) \leq \|\gs^m\kappa\|_{L^2}^2 (0) + c(m,n,p,\varepsilon,\delta,\lambda,\Lambda)  T
	\end{align*}
	for all $m\in\N_0$. 

From this we conclude using the upper bound on the length of the curve, $\|\gs^m \kappa\|_{L^1} \leq 
\Ell(f_t)^{\frac 1 2 } \|\gs^m\kappa\|_{L^2}^2 $, and Lemma \ref{lem:uniformbounds} that
	\begin{align*}
		\|\fs^m \kappa \|_{L^\infty} \leq c(m,n,p,\varepsilon,\delta,\lambda,\Lambda,f_0, T) 
	\end{align*}
	for all $m\in\N_0$. By an inductive argument as in the proof of {\cite[Theorem~3.1]{Dziuk2002}} also 
	\begin{align*}
	\|\fx^m \kappa \|_{L^\infty} \leq c(m,n,p,\varepsilon,\delta,\lambda,\Lambda,f_0, T) 
	\end{align*}
	holds for all $m\in\N_0$. These bounds imply that $f$ smoothly extends to $[0,T]\times \R/\Z$ and together with the short-time existence result even beyond $T$, from which we iteratively get long-time existence. 
	
	Now we draw our attention to the asymptotic behaviour of long-time solutions to the regularized gradient flow \ref{regflow}. For any such solution $f(t,\cdot)$ we can choose, after reparametrization by arc-length and an appropriate choice of translations, a subsequence $\tilde{f}(t_i,\cdot) - p_i$ that smoothly converges to a limit curve $f_\infty$ as $t_i \rightarrow \infty$. 
	
	By the standard interpolation inequality as stated in \eqref{eq:standInterpolation} 
	for $k=m+3$, the bound on the length of the curve \eqref{eq:uniformboundL}, and estimate \eqref{eq:proof8}, we infer 
	\begin{align*}
	\frac{d}{dt} \int_{\R/\Z} |\gs^m\kappa|^2 ds + c_0 \int_{\R/\Z} |\gs^m\kappa|^2 ds \leq c(m,n,p, \varepsilon, \delta,\lambda, f_0),
	\end{align*}
	where $c_0 = c_0 (p,\varepsilon,\delta,\lambda,f_0)>0$, which gives us the bound on $\|\gs^m\kappa\|_{L^2}^2 (t) \leq \|\gs^m\kappa\|_{L^2}^2 (0) + c(m,n,p,\varepsilon,\delta,\lambda,f_0)$. Together with the bound on the length \eqref{eq:uniformboundL}, Lemma~\ref{lem:uniformbounds}, and Lemma \ref{lem:gsidentities},  we get as a consequence
	\begin{align}\label{eq:boundgskappa}
		\|\partial_s^m\kappa\|_{L^\infty} + \|\gs^m\kappa\|_{L^\infty} \leq c(m,n,p,\varepsilon,\delta,\lambda,f_0)
	\end{align} 
	and hence also 
	\begin{align}\label{eq:boundgtgskappa}
		\|\gt (\gs^m\kappa)\|_{L^\infty} \leq c(m,n,p,\varepsilon,\delta,\lambda,f_0).
	\end{align}
	by Lemma \ref{lem:eveq2}.
	
	Next we define $u(t) = \|\partial_t f \|^2_{L^2} (t)$. Then we have  $u(t)\in L^1((0,\infty))$ by the energy identity \eqref{eq:energyidentity} and  $|u'(t)| \leq c(m,n,p,\varepsilon,\delta,\lambda,f_0)$ by \eqref{eq:uniformboundL}, \eqref{eq:boundgskappa}, and \eqref{eq:boundgtgskappa}. Thus we obtain that $u(t) \rightarrow 0$ as $t\rightarrow  \infty$, from which we conclude that $f_\infty$ is a critical point of $\E^p_{\varepsilon,\delta,\lambda}$. 
	
This completes the proof of Theorem \ref{thm:LongTimeExistenceRE}.

\section{A 	priori estimates for the regularized energies} \label{sec:APriori}

In this section we prove the essential a priori estimates that will allow to obtain a solution to \eqref{regflow} sending $\varepsilon, \delta \downarrow 0$ in Section \ref{sec:Convergence}. To this end, we first have to bring the evolution equations in a suitable form.

\subsection{The first variation of the energy}

We will once more calculate the first variation of the energy $\ER$. This time we will not decompose the direction into a normal and tangential part as later on we test the equation with a finite difference $\g(\cdot+h) - \g$ which need not be pointing in normal direction.  
We consider a smooth time dependent family of curves 
$
\g(t,x)=\g: [0,T) \times \RZ \rightarrow \mathbb R^n
$
and set $V = \partial_t \g.$
We obtain the following variant of \cite[Lemma 9]{Dziuk2002}.

\begin{lemma}[First variation of geometric quantities]
We have
\begin{enumerate}[label=(\roman*)]
	\item $\partial_t (ds) = \langle \tau, \partial_s V \rangle ds$ ,
	\item $\partial_t \partial_s - \partial_s \partial_t = - \langle \partial_s V, \tau \rangle \partial_s$,
	\item $\partial_t \tau = \nabla_s V$,
	\item $\nabla_t \kappa = (\partial_s^2 V)^\bot - 2 \langle \partial_s V, \tau \rangle \kappa - \langle \partial_s V, \kappa \rangle \tau  $,
	\item $\nabla_t \nabla_s \kappa = (\partial_s^3 V)^\bot - 3 \langle \partial_s^2 V, \tau\rangle \kappa   - 3 \langle \partial_s V, \tau\rangle (\partial_s \kappa)^\bot
 - \langle \nabla_s V, \kappa \rangle \kappa + |\kappa|^2 \nabla_s V$.
\end{enumerate}
\end{lemma}

\begin{proof}
Firstly, we observe that 
$$
 \partial_t |\g'| = \left \langle \frac {\g'}{|\g'|}, \partial_x V \right\rangle = \langle \tau, \partial_s V \rangle  |\g'|.
$$
Hence, from $ds = |\g'|dx$ we get
$$
 \partial_t (ds) = (\partial_t |\g'|) dx = \langle \tau, \partial_s V \rangle ds
$$
and
$$ 
\partial_t \tau = \frac {\partial_t \g'}{|\g'|} - \frac {\g'}{|\g'|} \langle \tau, \partial_s V \rangle = \partial_s V - \langle \tau, \partial_s V\rangle \tau = \nabla_s V.
$$
This proves parts (i) and (iii). Similarly,  differentiating $\partial_s = \frac {\partial_x}{|\g'|}$ yields
$$
 \partial_t \partial_s = \frac 1 {|\g']} \partial_t \partial_x - \langle \tau, \partial_s V\rangle \frac 1 {|\g'|}  \partial_x = \partial_s \partial_t - \langle \tau, \partial_s V\rangle  \partial_s.
$$
To get (iv), we calculate using the formulas (ii) and (iii)
\begin{align*}
 \partial_t \kappa 
 & = \partial_t \partial_s \tau = \partial_s \partial_t \tau  - \langle \tau ,\partial_s V \rangle \partial_s \tau 
 = \partial_s (\partial_s V - \langle \partial_s V, \tau \rangle \tau) - \langle \tau ,\partial_s V \rangle \kappa \\
 & = (\partial_s^2 V )^\bot - \langle \partial_s V, \kappa \rangle \tau - \langle \partial_s V, \tau \rangle \kappa  - \langle \tau ,\partial_s V \rangle \kappa\\
 & = (\partial_s^2 V)^\bot - 2 \langle \partial_s V, \tau \rangle \kappa - \langle \partial_s V ,\kappa \rangle \tau.
\end{align*}
Finally from 
\begin{align*}
 \nabla _t \partial_s \kappa &= (\partial_s \partial_t \kappa  - \langle  \partial_s V , \tau \rangle \partial_s \kappa)^\bot \\
 & = (\partial_s  (\partial_s^2 V - \langle \partial_s^2 V, \tau \rangle \tau - 2 \langle \partial_s V, \tau \rangle \kappa + \langle \partial_s V ,\kappa \rangle \tau)  - \langle \partial_s V, \tau  \rangle \partial_s \kappa)^\bot \\
& = ((\partial_s^3 V)^\bot - \langle \partial_s^2 V, \tau\rangle \kappa - 2 \langle \partial_s^2 V, \tau\rangle \kappa  - 2 \langle \partial_s V, \kappa \rangle \kappa - 2 \langle \partial_s V, \tau\rangle \partial_s \kappa
 \\ & \quad + \langle \partial_s V, \kappa \rangle \kappa - \langle\partial_s V, \tau \rangle \partial_s \kappa)^\bot 
 \\ &= (\partial_s^3 V)^\bot - 3 \langle \partial_s^2 V, \tau\rangle \kappa   - 3 \langle \partial_s V, \tau\rangle \nabla_s \kappa
 - \langle \partial_s V, \kappa \rangle \kappa
\end{align*}
together with
\begin{align*}
  \nabla_t ( \langle \partial_s \kappa, \tau \rangle \tau ) = \langle \partial_s \kappa, \tau \rangle  \nabla_t  \tau = \langle \partial_s \kappa, \tau \rangle \nabla_s V
 =-|\kappa|^2 \nabla_s V
\end{align*}
we get (v).
\end{proof}

Using this lemma with $\g_t = \g + t V $, we obtain for the first variation of $\F$
\begin{equation} \label{eq:FirstVariationF}
\begin{aligned}
 \delta_V \F(\g) &:= \partial_t \F (\g+ t V) |_{t=0}  \\ & = \int_{\RZ} \langle \nabla _s \kappa, \delta_V( \nabla _s \kappa) \rangle ds + \frac 12 \int_{\RZ} |\nabla_s \kappa|^2  \delta_V (ds) 
 \\ &= \int_{\RZ} \langle \nabla _s \kappa,  (\delta_V( \nabla _s \kappa)) ^\bot \rangle ds + \frac 12 \int_{\RZ} |\nabla_s \kappa|^2  \langle \tau, \partial_s V \rangle ds,
\end{aligned}
\end{equation}
where
$$
(\delta_V (\nabla_s \kappa))^\bot = (\partial_s^3 V)^\bot - 3 \langle \partial_s^2 V, \tau\rangle \kappa   - 3 \langle \partial_s V, \tau\rangle \nabla_s \kappa
 - \langle \nabla_s V, \kappa \rangle \kappa + |\kappa|^2 \nabla_s V. 
$$
Similarly,
\begin{equation} \label{eq:FirstVariationEpd}
\delta_V \E^p_\delta ( \g ) = \int_{\RZ} (|\kappa|^2 + \delta^2 )^{p-2} \langle \kappa, \delta_V \kappa \rangle + \frac 1 p \int_{\RZ} (|\kappa|^2 + \delta^2)^{\frac p 2} \langle \tau, \partial_s V \rangle ds 
\end{equation}
and
$$
\delta_V \kappa =  (\partial_s^2 V)^\bot - 2 \langle \partial_s V, \tau \rangle \kappa - \langle \partial_s V, \kappa \rangle \tau .
$$
The first variation of length of course is
\begin{equation} \label{eq:FirstVariationLength}
 \delta_V \Ell (\g) = \int_{\RZ} \langle \tau, \partial_s V \rangle  ds.
 \end{equation}
Let us bring the first variation of $\nabla_s \kappa$ in a more suitable form for our task, using again a slight variant of the $\ast$-notation introduced in \cite{Dziuk2002}.  
Together with $\nabla_s \kappa = \partial_s \kappa - \langle \partial_s \kappa, \tau \rangle \tau = \partial_s \kappa + |\kappa|^2 \tau $, we get 
$$
(\delta_V (\nabla_s \kappa))^\bot = (\partial_s^3 V)^\bot + \partial_s^2 V \ast \tau \ast \kappa   + \partial_s V \ast (\tau  \ast \partial_s \kappa + \kappa \ast \kappa + \kappa \ast \kappa \ast \tau). 
$$

\subsection{Higher regularity}

We will now prove that solutions of the negative gradient flow for the regularized energies $\ER$ are more regular, using a finite difference method as in \cite{Brasco2017}. We measure this additional regularity using the Besov-Nikolski space $B^{s}_{p, \infty}(\RZ)$ for $s \in (0,1)$ and $p\in [1, \infty)$, which consists of all functions $u \in L^p(\RZ)$ for which the semi-norm
$$
 |u|_{B^{s}_{p,\infty}} := \sup_h \frac {\|u_h - u\|_{L^p}}{|h|^{s}}
$$
is finite. Here, $u_h(x):=u(x+h)$ denotes the by the value $h$ shifted function. This space is equipped with the norm 
$$
 \|u\|_{B^{s}_{p,\infty}} = \|u\|_{L^p} + |u|_{B^{s}_{p,\infty}}.
$$

\begin{theorem}[Higher regularity] \label{thm:HigherRegularity}
Let $\g:\R/\Z\rightarrow \R^n$ be a smooth curve parametrized by arc-length which satisfies
the equation
\begin{equation} \label{eq:QuasiSolution}
 \nabla_{L^2} \ER(\g) = g
\end{equation}
for a $g \in L^2(ds)$ and an $0< \varepsilon < 1$. Then
$$
 \varepsilon \|\partial_s^3\g\|_{B^{\frac 14}_{2, \infty}} ^2 + \|\partial_s^2 \g\|_{B^{\frac 1{2p} } _{p,\infty}}^p \leq C (1+ \|g\|_{L^2})
$$
for some constant $C= C(p,\lambda,\ER(\g))$.
\end{theorem}

We note that Fenchel's together with H\"older's inequality give
$$
 2 \pi \leq \int_{\RZ} |\kappa|ds \leq \Ell(\g)^{1- \frac 1p} \Big(\int_{\RZ} |\kappa|^p ds\Big)^{\frac 1p} \leq  \Ell(\g)^{1- \frac 1p}(p\ER(\g) )^{\frac 1p},
$$
which implies
$$
 \Ell(\g)  \geq \frac {(2 \pi)^{\frac {p} {p-1} }} {(p\ER(\g))^{\frac 1 {p-1}}}.
$$
Since trivially 
$$
 \Ell(\g) \leq \frac {\ER(\g)}{\lambda},
$$
we see that the length of the curve is bounded from above and from below away from zero if the regularized energy is controlled.

We note that using H\"older's inequality we get
\begin{equation} \label{eq:BoundEnergy}
 \frac \varepsilon 2 \int_{\RZ} |\nabla_s \kappa|^2 ds + \frac 1p \Ell(\g)^{\frac  p 2 -1 } (\int_{\RZ} |\kappa|^2 ds)^{\frac p2} \leq \varepsilon\F(\g) + \EPD (\g)  \leq \ER(\g)
\end{equation}
as $p \geq 2.$ 
From $\nabla_{s} \kappa = \partial_{s} \kappa + |\kappa|^2 \tau$ together with Gagliardo-Nirenberg's interpolation inequality (cf. Lemma~\ref{lem:interpolationkappa}) and Young's inequality, we get
\begin{align*}
	 \|\partial_{s} \kappa\|_{L^2} &\leq  C (\|\nabla_{s} \kappa\|_{L^{2}} + \|\kappa\|^{2}_{L^4})
	\leq C (  \|\nabla_{s} \kappa\|_{L^{2}} +  \|\nabla_{s} \kappa\|^{\frac 1 2}_{L^{2}}
	\| \kappa \|^{\frac 32}_{L^2} )
	\\ &\leq C ( \|\nabla_{s} \kappa\|_{L^{2}} + \|\kappa\|_{L^2}^3) \leq C ( \|\nabla_{s} \kappa\|_{L^{2}} + 1).
\end{align*}
Hence, \eqref{eq:BoundEnergy} implies
$
  \varepsilon \|\partial_{s} \kappa\|^2_{L^2} \leq C,
$
where the constant $C>0$ only depends on the energy $\ER(\g).$ 
From Equation \eqref{eq:QuasiSolution} we get  
$$
 \nabla_{L^2} \ER(\g_{h}) = g_{h} \quad \text{and} \quad \nabla_{L^{2}} \ER(\g) =g ,
$$
where for a function $u$ we set $u_{h}(x):= u(x+h)$. We test these equations with $V=\g_h - \g$, subtract the results and integrate over $\RZ$ to get
\begin{equation} \label{eq:WF}
\begin{aligned}
  \varepsilon \delta_{V} (\F(\g_{h}) - \F(\g))  &+ \delta_{V} (\EPD (\g_{h}) - \EPD (\g)) + \lambda (\delta_{V} (\Ell(\g_{h}) - \Ell(\g))  \\ & =\int_{\RZ} (g_{h}-g) (\g_{h} -\g) ds.
\end{aligned}
\end{equation}
We will estimate this further using that
\begin{equation} \label{eq:Hoelder}
 \hol_{\frac 1 2} \tau \leq C \|\kappa\|_{L^2}  \qquad  \hol_{\frac 1 2} \kappa \leq C \|\partial_{s}\kappa\|_{L^2}
\end{equation}
which can easily seen to be true observing that the fundamental theorem of calculus together with H\"older's inequality gives for a differentiable function $u: \RZ \rightarrow \mathbb R^n$ and $x,h \in \mathbb R$ with $|h|< 1$,
$$
 |u(x+h)- u(x)| = \Big|\int_0^h u'(x+t) dt \Big| \leq \int_0^h |u'(x+t)| dt \leq |h|^{\frac 1 2} \|u'\|_{L^2}.
$$
Hence, we also have 
\begin{equation} \label{eq:EstimateV}
 \|\partial_s V\|_{L^\infty} \leq C |h|^{\frac 1 2}\|\kappa\|_{L^2}  \qquad  \|\partial_s^2 V \|_{L^\infty} \leq C |h|^{\frac 1 2}\|\partial_{s}\kappa\|_{L^2}.
\end{equation}
We start discussing the terms coming from the first variation of $\F$ in \eqref{eq:WF}.

\begin{lemma} \label{lem:APF}We have 
$$
 \varepsilon \delta_V (\F (\g_h) - \F (\g)) = \varepsilon \int_{\RZ} |\partial_s^3 \g_h - \partial_s^3 \g|^2 ds + R_{\F}
$$
with $ |R_{\F}| \leq C |h|^{\frac 12}.$
\end{lemma}

\begin{proof}
To prove this claim, we decompose 
\begin{align*}
& \varepsilon\Big( \int_{\RZ} \langle \nabla_s \g_h, \delta_V (\nabla_s \g_h)\rangle ds - \int_{\RZ} \langle \nabla_s \g, \delta_V (\nabla_s \g)\rangle  ds \Big) 
\\ &= \varepsilon \int_{\RZ} \langle \partial_s^3 \g_h+ |\kappa_h|^2 \tau_h, \partial_s^3 V + \partial_s^2 V \ast \tau_h \ast \kappa_h \\ & \qquad \qquad+ \partial_s V \ast (\tau_h \ast \partial_s\kappa_h + \kappa_h \ast \kappa_h + \kappa_h \ast \kappa_h \ast \tau_h) \rangle ds
\\ & \quad -  \varepsilon \int_{\RZ} \langle \partial_s^3 \g+ |\kappa|^2 \tau, \partial_s^3 V + \partial_s^2 V \ast \tau \ast \kappa + \partial_s V \ast (\tau \ast \partial_s\kappa + \kappa \ast \kappa + \kappa \ast \kappa \ast \tau) \rangle ds
\\ & = \varepsilon  \int_{\RZ}|\partial^3_s \g_h - \partial^3_s \g|^2 ds +\varepsilon( I + II  + III ), 
\end{align*}
where
\begin{align*}
 I &= \int_{\RZ} \langle |\kappa_h|^2 \tau_h - |\kappa|^2 \tau, \partial_s^3 V \rangle  ds ,\\
 II & = \int_{\RZ} \langle\partial_s^3 \g_h,  \partial_s^2 V \ast \kappa_h \ast \tau_h + \partial_s V \ast (\tau_h \ast \partial_s\kappa_h + \kappa_h \ast \kappa_h + \kappa_h \ast \kappa_h \ast \tau_h ) \rangle ds
 \\ & \quad - \int_{\RZ} \langle \partial_s^3 \g, \partial_s^2 V \ast \kappa \ast \tau + \partial_s V \ast (\tau \ast \partial_s\kappa + \kappa \ast \kappa + \kappa \ast \kappa \ast \tau ) \rangle ds,
 \\ 
 III & =  \int_{\RZ} \langle |\kappa_h|^2 \tau_h,\partial_s^2 V \ast \kappa_h \ast \tau_h + \partial_s V \ast (\tau_h \ast \partial_s\kappa_h + \kappa_h \ast \kappa_h + \kappa_h \ast \kappa_h \ast \tau_h ) \rangle ds
 \\ & \quad - \int_{\RZ} \langle |\kappa|^2  \tau , \partial_s^2 V \ast \kappa \ast \tau + \partial_s V \ast (\tau \ast \partial_s\kappa + \kappa \ast \kappa + \kappa \ast \kappa \ast \tau ) \rangle ds.
\end{align*}

For the term I we see
\begin{align*}
 |I| \leq \|\partial_s^3 V\|_{L^2} \||\kappa_h|^2 \tau_h - |\kappa|^2 \tau \|_{L^2}
 \leq 2 \|\partial_s \kappa\|_{L^2 } \||\kappa_h|^2 \tau_h - |\kappa|^2 \tau \|_{L^2}.
\end{align*}
The equality 
$$
 |\kappa_h|^2 \tau_h - |\kappa|^2 \tau =  (|\kappa_h| - |\kappa|) |\kappa_h| \tau_h + |\kappa| (|\kappa_h| -| \kappa|) \tau_h + |\kappa|^2 (\tau_h - \tau)
$$
together with \eqref{eq:Hoelder} and Gagliardo-Nirenberg interpolation estimates (cf. Lemma \ref{lem:interpolationkappa}) imply that 
\begin{align*}
 \||\kappa_h|^2 \tau_h - |\kappa|^2 \tau \|_{L^2} &\leq |h|^{\frac 1 2}(2  (\hol_{\frac 1 2} \kappa ) \|\kappa\|_{L^2}+ (\hol_{\frac 1 2} \tau) \|\kappa\|_{L^4}^2  ) 
 \\ & \leq C |h|^{\frac 1 2} (\|\partial_s \kappa\|_{L^2} \|\kappa\|_{L^2} + \|\kappa\|_{L^2}  \|\kappa\|_{L^4}^2)
 \\ & \leq    C |h|^{\frac 1 2} (\|\partial_s \kappa \|_{L^2}   \|\kappa\|_{L^2} +  \|\partial_s \kappa\|_{L^2}^{\frac 1 2}\|\kappa\|^{1 +\frac 32}_{L^2})
\\ & \leq C |h|^{\frac 1 2} (\|\partial_s \kappa\|_{L^2} + 1)
\end{align*}
as $\|\kappa\|_{L^2}$ is bounded.
Hence,
$$
 \varepsilon |I| \leq C |h|^{\frac 1 2} \varepsilon (\|\partial_s \kappa\|_{L^2}^2 + \|\partial_s \kappa\|_{L^2} ) \leq C |h|^{\frac 12 } \varepsilon(\|\partial_s \kappa\|_{L^2}^2 +1) \leq C |h|^{\frac 12},
$$
where we applied \eqref{eq:BoundEnergy}.
To see that the same estimate holds for the terms in II, we observe by \eqref{eq:EstimateV} 
\begin{align*}
 \int_{\RZ} \langle\partial_s^3 \g, \partial^2 V \ast \kappa \ast \tau\rangle  ds
 \leq \|\partial_s \kappa \|_{L^2} \|\kappa\|_{L^2 }\|\partial^2 V\|_{L^\infty}
 \leq C |h|^{\frac 1 2 } \|\partial_s \kappa\|_{L^2}^2 
\end{align*}
and
\begin{align*}
  \int_{\RZ} \langle \partial_s^3 \g, \partial_s V \ast &( \tau \ast \partial_s\kappa + \kappa \ast \kappa + \kappa \ast \kappa \ast \tau) \rangle ds \\ &\leq C \|\partial_s \kappa\|_{L^2} \|\partial_s V\|_{\infty} (\|\partial_s \kappa\|_{L^2} + \|\kappa\|^2_{L^4}) \
  \\ &\leq C |h|^{\frac 1 2} (\|\partial_s \kappa\|_{L^2}^2 \|\kappa\|_{L^2} + \|\partial_s \kappa\|_{L^2}^{1+\frac 1 2} \|\kappa\|^{1+\frac 3 2}_{L^2}) \\
  & \leq C |h|^{\frac 12} (\|\partial_s \kappa\|_{L^2}^2 +1) ,
\end{align*}
where again we used \eqref{eq:BoundEnergy} and the interpolation estimates stated in Lemma \ref{lem:interpolationkappa}.
As the other terms in $II$ envolving $\g_h$ instead of $\g$ can be estimated in precisely the same way, we obtain
$$
  \varepsilon | II | \leq C |h|^{\frac 1 2} \varepsilon (\|\partial_s \kappa\|_{L^2}^2 +1) \leq C |h|^{\frac 1 2}.
$$
Similarly,
\begin{align*}
| III | &\leq C ( \|\kappa\|^3_{L^3}\|\partial_s^2 V\|_{L^\infty} +  \|\partial_s\kappa\|_{L^2} \|\kappa\|_{L^4}^2 \|\partial_s V\|_{L^{\infty}} + \|\kappa\|_{L^4}^4 \|\partial_s V\|_{L^{\infty}} ) 
\\ & \leq C |h|^{\frac 1 2} (\|\kappa\|_{L^3}^{3} \|\partial_s \kappa\|_{L^2} + \|\partial_s\kappa\|_{L^2}^{1+\frac 12} \|\kappa\|_{L^2}^{1+\frac 32}+  \|\partial_s\kappa\|_{L^2} \|\kappa\|_{L^2}^4) \\
& \leq C |h|^{\frac 12} (\|\partial_s \kappa\|^{2 }_{L^2}  + 1)
\end{align*}
Hence, also
$$
   \varepsilon | III |  \leq C |h|^{\frac 1 2}.
$$
As finally
$$
 \int_{\RZ} |\nabla_s \kappa|^2 \langle \tau, \partial_s V \rangle ds
 \leq C |h|^{\frac 1 2} \|\nabla_s \kappa \|_{L^2}^2 \|\kappa\|_{L^2} \leq C |h|^{\frac 1 2} \|\nabla_s \kappa \|_{L^2}^2,
$$
this gives
$$
 \varepsilon \delta_V (\F (\g_h) - \F (g)) = \varepsilon \int_{\RZ} |\partial_s^3 \g_h - \partial_s^3 \g|^2 ds + R_{\F}
$$
with
$$
 |R_{\F}| \leq C |h|^{\frac 12}.
$$
\end{proof}

The terms in \eqref{eq:WF} containing $E^{(p)}_\delta$ can be estimated as follows.

\begin{lemma} \label{lem:APEP}We have 
$$
 \delta_ V \EPD (\g_h) - \delta_V \EPD(\g) =  \int_{\RZ}  \langle (|\kappa_h|^2 + \delta^2)^{\frac {p-2} 2}  \kappa_h - (|\kappa|^2 + \delta^2)^{\frac {p-2} 2} \kappa , \partial^2_s V \rangle ds   +  \mathcal R_{\EPD}
$$
with $|\mathcal R_{\EPD}| \leq C |h|^{\frac 12}$.
\end{lemma}

\begin{proof}
We decompose
\begin{align*}
\int_{\RZ} (|\kappa_h|^2 + \delta^2)^{\frac {p-2} 2} & \langle \kappa_h, \delta_V (\kappa_h) \rangle - (|\kappa|^2 + \delta^2)^{\frac {p-2} 2} \langle \kappa , \delta_V (\kappa) \rangle ds
\\ & = \int_{\RZ} (|\kappa_h|^2 + \delta^2)^{\frac {p-2} 2}  \langle \kappa_h , \partial_s^2 V + \partial_s V \ast \kappa_h \ast \tau_h\rangle  
\\ & \quad - (|\kappa|^2 + \delta^2)^{\frac {p-2} 2} \langle \kappa , \partial_s^2 V + \partial_s V \ast \kappa \ast \tau \rangle ds \\
\\ & =  \int_{\RZ} \langle (|\kappa_h|^2 + \delta^2)^{\frac {p-2} 2}  \kappa_h - (|\kappa|^2 + \delta^2)^{\frac {p-2} 2} \kappa , \partial^2_s V \rangle ds + I
\end{align*}
where this time
$$
I =  \int_{\RZ} (|\kappa_h|^2 + \delta^2)^{\frac {p-2} 2}  \langle \kappa_h, \partial_s V \ast \kappa_h \ast \tau _h \rangle ds -  \int_{\RZ} (|\kappa|^2 + \delta^2)^{\frac {p-2} 2}  \langle \kappa, \partial_s V \ast \kappa \ast \tau \rangle ds. 
$$
Observe that 
$$
 | I | \leq C\|\partial_sV\|_{L^\infty} \int_{\RZ} (|\kappa|^2 + \delta ^2 )^{\frac {p-2} 2} |\kappa|^2 ds  \leq C \|\kappa\|_{L^2 } \E^{(p)}_{\delta} (\g) \leq C |h|^{\frac 12}.
$$
Furthermore, we have 
$$
 |\tfrac 1p \int_{\RZ}(|\kappa|^2 + \delta^2)^{\frac {p} 2 } \langle \partial_s V, \tau\rangle ds| \leq \E^{(p)}_{\delta} (\g) \|\partial_s V\|_{L^\infty} \leq C |h|^{\frac 1 2}.
$$
This proves the claim.
\end{proof}

\begin{lemma} \label{lem:APL}We have 
	$$
	\delta_V \Ell(\g_h) - \delta_V \Ell(\g) = \mathcal R_{\Ell}
	$$
	with $|\mathcal R_{\Ell}| \leq C |h|^{\frac 12}$.
\end{lemma}
\begin{proof}
	The first variation of length \eqref{eq:FirstVariationLength} together with \eqref{eq:EstimateV} leads to
	$$
	\delta_V \Ell(\g_h) - \delta_V \Ell(\g) \leq C \|\partial_s V\|_{L^\infty} \leq C |h|^{\frac 12},
	$$
	which proves the statement.
\end{proof}

The following lemma tells us, that the main part in Lemma \ref{lem:APEP} has the similar monotonicity properties as the $p$-Laplace operator. 
\begin{lemma} \label{lem:MonotonicityEP}
We have 
$$
\int_{\RZ}  \langle (|\kappa_h|^2 + \delta^2)^{\frac {p-2} 2}\kappa_h - (|\kappa|^2 + \delta^2)^{\frac {p-2} 2} \kappa , \partial^2_s V \rangle ds  \geq  \frac 1 {4^{p-1}} \int_{\RZ} |\kappa_h - \kappa|^p ds.
$$
\end{lemma}

\begin{proof}
For vectors $w,v$  we set $w_t = v + (w-v) t$. Using the fundamental theorem of calculus, we obtain 
\begin{align*}
\langle (|w|^2 +\delta^2)^{\frac {p-2} 2} w &- (|v|^2+ \delta^2)^{\frac{p-2}2} v , w-v \rangle
\\ &= \int_0^1  (p-2) (|w_t|^2 + \delta^2)^{\frac {p-4}{2}}
\langle w_t, (w-v)\rangle ^2 dt
\\ & \quad +  \int_{0}^1 (|w_t|^2 + \delta^2)^{\frac {p-2}{2}} |w-v|^2  dt
\\ & \geq  \int_{0}^1 (|w_t|^2 + \delta^2)^{\frac {p-2}{2}} |w-v|^2 dt.
\end{align*}
If we now assume that $|v|\geq |w|$, we get
for $t \in [0,\tfrac 1 4]$
\\
$$
 |w_t| \geq |v| - t (|w|+|v|) \geq |v| - \frac 12 |v| = \frac { |v|} 2 \geq \frac 14 |w-v|.  
$$
So,
\begin{align*}
\langle (|w|^2 +\delta^2)^{\frac {p-2} 2} w - (|v|^2+ \delta^2)^{\frac{p-2}2} v , w-v \rangle &\geq \int_0^{\frac 1 4}  (|w_t|^2 + \delta^2)^{\frac {p-2}{2}} |w-v|^2 dt  \\ & \geq 
 \frac 1 {4^{p-1}}|w-v|^p 
\end{align*}
and hence
$$
\int_{\RZ} \langle (|\kappa_h|^2 + \delta^2)^{\frac {p-2} 2}  \kappa_h -  (|\kappa|^2 + \delta^2)^{\frac {p-2} 2} \kappa , \partial^2 _s V \rangle ds  \geq  \frac 1{4^{p-1}}\int_{\RZ} |\kappa_h - \kappa|^p ds.
$$
\end{proof}

Now we can prove Theorem \ref{thm:HigherRegularity}.
\begin{proof} [Proof of Theorem \protect{\ref{thm:HigherRegularity}}]

Note that 
$$
 \int_{\RZ} (g_h - g) Vds \leq C \|g\|_{L^2} |h|^{\frac 12}.
$$

Using equation \eqref{eq:QuasiSolution} together with the Lemmata \ref{lem:APF} and \ref{lem:APEP} and the estimate above, we get
\begin{multline*}
 \varepsilon \int_{\RZ} |\partial^3_s \g_h - \partial_s^3 \g|^2 ds +  \int_{\RZ} \langle (|\kappa_h|^2 + \delta^2)^{\frac {p-2} 2}   \kappa_h  -(|\kappa|^2 + \delta^2)^{\frac {p-2} 2}   \kappa , \partial_s^2 V \rangle ds \\  \leq C |h|^{\frac 1 2} (\|g\|_{L^2} + 1 ).
\end{multline*}
Together with Lemma \ref{lem:MonotonicityEP} and the definition of the Besov spaces we hence obtain 
$$
 \varepsilon \|\partial_s^3 \g\|_{B^{\frac 1 4}_{2,\infty}}^2 + \|\kappa\|^p_{B^{\frac 1{2p}}_{p, \infty}} \leq C (\|g\|_{L^2} + 1 ).
$$
\end{proof}

Let us list some immediate consequences for solutions to the gradient flow for the regularized energies. Firstly we observe that the highest order part grows slower than the trivial bound $\varepsilon \F (\g) < C$ might suggest:

\begin{corollary} \label{cor:FastDecayF}
For any solution $\rg:[0,T) \times \RZ \rightarrow \R^n$ of Equation \eqref{regflow} we get
$$
  \varepsilon \int_0^T \int_{\RZ} |\partial^3_s \rg| ^2 ds dt \leq \varepsilon^{\frac 1 5} C (T+1)
$$ 
where the constant $C=C(p,\lambda, \ER(\rg_0))>0$ depends only on $p, \lambda,$ and $\ER (\rg_0)$.
\end{corollary}

\begin{proof}
From Theorem \ref{thm:HigherRegularity} we ge by integrating over time and using H\"older's inequality
\begin{equation}
\begin{aligned}
 \varepsilon \int_{0}^T \|\partial^3_s \rg \|^2_{B^{\frac 1 4}_{2,\infty}}  dt 
 & \leq C\int_0^T (\|\partial_t \rg\|_{L^2} +1 ) ds
 \\ &  \leq C  T^{\frac 1 2}  \Big(\int_{0}^T \int_{\RZ} |\partial_t \rg |^2 dt \Big)^{\frac 1 2}  + C T
 \\ & \leq C (T^{\frac 1 2} + T).
 \end{aligned}
\end{equation}
Together with the interpolation estimate (cf. \cite{Triebel1983})
$$
 \|\partial_s^3 \rg\|_{L^2} \leq C  \|\partial_s^2 \g\|^{\frac 45}_{B^{1+\frac 14}_{2,\infty}} \|\kappa\|^{\frac 15}_{L^2}, %
$$
this implies using Young's inequality 
\begin{align*}
 \varepsilon  \int_{0}^T \int_{\RZ} | \partial_s^3 \rg |^2 ds dt & \leq \varepsilon^{\frac 15} C (\int_0^T  \varepsilon ^{\frac 45}\|\partial_s^3 \rg\|^{\frac 85}_{B^{\frac 14}_{2, \infty}} dt  +1) 
 \\ & \leq \varepsilon^{\frac 1 5} C (\int_0^T  \varepsilon \|\partial_s^3 \rg\|^{2}_{B^{\frac 14}_{2, \infty}} dt  +1)
 \leq C \varepsilon^{ \frac 15} (T+1).
\end{align*}
\end{proof}

Furthermore, just integrating over the estimate in Theorem \ref{thm:HigherRegularity}. using H\"older's inequality, and the uniform bound on the $L^p$ norm of $\kappa$ we get the following Corollary.

\begin{corollary} \label{cor:EstimateKappa}
For any solution $\rg:[0,T) \times \RZ\rightarrow \R^n$  of Equation \eqref{regflow} we get
$$
  \int_0^T  \|\kappa\|^p_{B^{\frac 1 {2p}}_{p,\infty}}  dt \leq C (T+1)
$$ 
where the constant $C=C(p,\lambda, \ER(\rg_0))>0$ depends only on $p, \lambda,$ and $\ER (\rg_0)$.
\end{corollary}

\section{Convergence to solutions} \label{sec:Convergence}
\subsection{The case of smooth initial data}

We now show the following version of Theorem \ref{thm:LongTimeExistenceRElastic} for smooth initial data.

\begin{theorem} \label{thm:LongTimeExistenceRElasticSmoothData}
Given any regular closed curve  $\g_0: \RZ \rightarrow \mathbb R^n$ of class $C^\infty$ parametrized with constant speed, there is a family of regular curves $\g: [0,\infty) \times \RZ \rightarrow \mathbb R^n$, $\g \in H^1([0,\infty), L^2(\RZ, \R^n)) \cap L^\infty([0,\infty), W^{2,p}(\RZ, \R^n)) \cap C^{\frac 12} ([0,\infty), L^2 (\RZ, \mathbb R^n ))$ solving the initial value problem 
\begin{equation*}
	\begin{cases}
	\partial^\bot_t \g & = - \nabla_{L^2} \E^{(p)}(f) \\
	\g(0, \cdot ) &= \g_0
	\end{cases}
\end{equation*}
in the weak sense, i.e. for all $V\in C^\infty_c ([0, \infty) \times \RZ, \R^n)$ we have 
\begin{equation*}
 \int_0^\infty \int_{\R/\Z} \langle \partial_t^\bot \g, V \rangle ds dt = - \int_0^\infty (\delta_V \E^{(p)} (\g)) dt.
\end{equation*}
Furthermore, this solution satisfies
$$
 \|\partial_t \g\|_{L^2((0,  \infty) \times \RZ) } \leq C,  \qquad   \|\kappa\|_{L^p((0, T), B^{\frac 1 {2p}}_{p, \infty})}  \leq C (T+1), 
$$
and
$$
 \|\g_{t_1} - \g_{t_0}\|_{L^2} \leq C |t_1 - t_0|^{\frac 1 2}
$$
for all $t_0,t_1 \in [0,\infty)$, where $C>0$ only depends on $\EP (\g_0)$, $p$,  and $\lambda$.
\end{theorem}

To construct this solution, we take solutions $\rg$ to the gradient flow of the regularized energies \eqref{regflow} with $0 < \varepsilon, \delta \leq 1$. We will now carefully reparametrize this family such that each curve is parametrized by constant speed.
Since $\rg$ is long-time solution for the gradient flow of $\ER$, we get 
\begin{equation*}
 \int_0^\infty \int_{\RZ} |\partial_t \rg|^2 dsdt \leq\ER(\rg_0).
 \end{equation*}
We now reparametrize the solution setting by constant speed setting
$$
 \trg(t,x) = \rg (t, \sigma_t^{\varepsilon, \delta } (x))
$$	
where $\sigma_t^{\varepsilon, \delta}$ is the inverse of the function  $\phi_t ^{\varepsilon, \delta } (x) =\frac 1 {\Ell(\rg)} \int_{x^{\varepsilon, \delta}}^x |(\rg_t (y))'| dy $.
Of course these reparametrized solutions solve the equation
$$
 \partial_t^{\bot}  \trg = - \nabla_{L^2} \ER (\trg) ,
$$
but they also satisfy 
\begin{align}
|(\trg) '| &= \Ell(\trg),  \label{eq:ConstantSpeed}\\
\partial_t^\bot \trg (t,0) &= \partial_t \trg (t,0), \label{eq:FullEquationAtAPoint}
\end{align}
and
\begin{equation} \label{eq:trg0}
 \int_0^\infty \partial_t  \trg(t,0) dt \leq C.
\end{equation}
We will now show that these furthermore have the following properties.

\begin{lemma} \label{lem:PropApproxSolutions}
The reparametrized solutions $\trg$ are parametrized by constant speed, solve the equation $\partial_t ^{\bot} \trg = - \nabla_{L^2} \ER (\trg),$ 
and satisfy the following estimates
\begin{enumerate} 
 \item $ \varepsilon \int_0^T \int_{\RZ} |\nabla_s \kappa|^2 ds dt \leq C \varepsilon^{\frac 1 5}(T+1),$ 
 \item  $\int_0^T \|\kappa\|^p_{B^{\frac 1 {2p}}_{p,\infty}} dt \leq C (T+1)$,
 \item $\int_{0}^\infty \int_{\RZ} |\partial_t \trg|^2 dxdt \leq C,$
  \item $\|\trg_{t_1} - \trg_{t_0} \|_{L^2}  \leq C |t_1-t_0|^{\frac 12}$ for all  $t_1, t_0 \in [0,\infty)$ with $|t_1 -t_0| \leq 1, $
\end{enumerate}
where $C>0$ is a constant only depending on the initial energy $\ER(\g_0)$.

\end{lemma}

\begin{proof}

The estimate (1) and (2)  follow directly from Corollary \ref{cor:FastDecayF} and Corollary \ref{cor:EstimateKappa}. Differentiating Equation \eqref{eq:ConstantSpeed} we get
$$
\Ell (\trg) \partial_t \Ell (\trg) = \langle (\trg)', \partial_x \partial_t  \trg \rangle
= - \Ell (\trg)^2 \langle \kappa, \partial_t  \trg \rangle + \partial_x \langle (\trg)', \partial_t  \trg \rangle
$$
and hence, using that the length of the curves is bounded from below and above,
$$
|\partial_x \langle (\trg)', \partial_t  \trg \rangle| \leq  C ( | \partial_t \Ell(\trg)| + |\langle \kappa, \partial^\bot_t \trg\rangle |).
$$
Integrating over this estimate, we get by applying the fundamental theorem of calculus, Hölder's inequality and  \eqref{eq:FullEquationAtAPoint} 

\begin{align*}
| \langle  (\trg)', \partial_t  \trg \rangle | &\leq C (|\partial_t \Ell (\trg) | + \|\kappa\|_{L^2} \|\partial_t^{\bot}  \trg \|_{L^2} )  + 
|\langle  (\trg)' (t,0), \partial_t  \trg (t,0)\rangle |  
\\ & \leq  C (|\partial_t \Ell (\trg) | +\|\kappa\|_{L^2} \|\partial_t^{\bot}  \trg \|_{L^2}  ) . 
\end{align*}
Combined with
$$
| \partial_t \Ell |= \left| \int_{\RZ} \langle \kappa, \partial_t^{\bot} \trg \rangle ds \right| \leq \| \kappa \|_{L^2} \|\partial_t^{\bot} \trg \|_{L^2} \leq C  \|\partial_t^{\bot} \trg \|_{L^2},
$$
this gives
$$
 | \partial_t ^{T} \trg | \leq C ( \|\partial_t^{\bot} \trg\|_{L^2} + |\kappa| |\partial_t^{\bot}  \trg |).
$$
Integrating over space and time and using Hölder's inequality, we hence obtain
$$
 \int_{0}^\infty \int_{\RZ} |\partial_t ^T \trg|^2 ds dt \leq C  \int_{0}^\infty \int_{\RZ} |\partial_t ^\bot \trg|^2 ds dt \leq C.
$$
As $ds = \Ell dx$ and the length is bounded from below, this proves property (3).

We can derive the Hölder estimate (4) using a standard estimate for $L^2$ gradient flows. Differentiating the quantity $\int_{\RZ} |\trg_{t} - \trg_{t_0}|^2 dx$ for a fixed time $t_0 \in (0, \infty)$,  we get  that
\begin{align*}
   \frac d {dt} \| \trg_{t} - \trg_{t_0}\|_{L^2}^2  & = 2 \int_{\RZ} \langle \trg_t - \trg_{t_0}, \partial_t \trg \rangle dx  
   \\ & \leq 2  \|\trg_t - \trg_0\|_{L^2}  \|\partial_t \trg\|_{L^2}.
\end{align*}
Hence, by the fundamental theorem of calculus
$$
 \|\trg_t - \trg_{t_0}\|_{L^2} \leq C |t-t_0|^{\frac 1 2}.
$$
So also (4) is proven.
\end{proof}

It is now straightforward to prove convergence of the solutions $\trg$ to a weak solution that has all the properties mentioned in Theorem \ref{thm:LongTimeExistenceRElasticSmoothData}. 

\begin{proof}[Proof of Theorem \protect{\ref{thm:LongTimeExistenceRElasticSmoothData}}]
Let $\varepsilon_n, \delta_n \rightarrow 0$ and let us set $\gn = \tilde \g^{\varepsilon_n, \delta_n}$. After chosing an appropriate subsequence, we can assume that $\partial_t \gn$ converges to $\partial_t \g$ weakly in $L^2$. Using that the solutions $\gn$ are uniformly bounded in $L^\infty W^{2,p}$, the compact embedding $W^{2,p} \hookrightarrow L^2$ and a standard diagonal sequence argument, we can furthermore assume after going to a subsequence that $\gn$ converges in $L^2$ for all times $t\in \mathbb Q \cap (0, \infty)$. Due to the uniform control of the Hölder constant  (4) in Lemma \ref{lem:PropApproxSolutions},  this subsequence then also converges locally in $L^\infty L^2$ to $\g: [0,\infty) \times \RZ \rightarrow \mathbb R^n.$ Interpolating once more, using that the $W^{2,p}$-norm of the curves is uniformly bounded, we also get convergence in  $L^\infty W^{1, \infty}$ locally in time.
If we integrate over the interpolation estimate
$$
 \|\partial_s ^2 (\gn - \g^{(m)}) \|_{L^p} \leq C (\|\partial_s^2 \gn- \partial_s^2 \g^{(m)}\|_{B^{\frac 1 {2p}}}^{\theta} \|\tau_n - \tau_m\|^{1-\theta}_{L^p} + \|\tau_n - \tau_m\|_{L^p} )
$$
and use property (2) of Lemma \ref{lem:PropApproxSolutions}, we see that
$$
 \int_{0}^T \int_{\RZ} |\kappa_n - \kappa_m|^p ds dt \leq C (T+1) \|\tau_n - \tau_m\|^p_{L^\infty((0,T), L^p)} \rightarrow 0
$$
as $n$ and $m$ go to $\infty$.
This implies that the $\gn$ even converge in $L^pW^{2,p}$ to $\g$ locally in time. Furthermore, we know from property (1) of Lemma \ref{lem:PropApproxSolutions} that
$$
 \varepsilon_n \int_0^T \int_{\RZ} |\nabla_s \kappa|^2 ds dt \leq C \varepsilon^{\frac 1 5}_n  \rightarrow 0.
$$
From the evolution equation we see that for all test functions $ V\in C^\infty_c (\RZ, \mathbb R^n)$ we have 
\begin{equation} \label{eq:WeakEquation1}
 \int_0^\infty \int_{\RZ}  \langle \partial_t^{\bot} \gn, V \rangle dx dt = \int_0^\infty ( \varepsilon \delta_V\F (\gn) + \delta_V\EPD (\gn)
 + \lambda \delta_V \Ell(\gn) ) dt.
\end{equation}
As $\partial_t \gn$ converges weakly to $\partial_t \g$ in $L^2$, $\tau_n$ converges strongly to $\tau$ in $L^\infty$, and $\partial_t ^\bot \gn = \partial_t \gn - \langle \partial_t \gn, \tau_n \rangle \tau_n$, we get
$$
 \int_0^\infty \langle \partial_t^\bot \gn, V \rangle dx dt \rightarrow \int_0^\infty \langle \partial_t^\bot \g, V \rangle dx dt
 $$
 as $n$ goes to $\infty.$ The first variation of $\varepsilon \F$ is given by
 $$
 \varepsilon \delta_V \F (\gn) = \varepsilon \int_{\RZ} \langle\nabla_s \kappa_n, \delta_{V} (\nabla_s \kappa_n)\rangle ds + \frac \varepsilon 2 \int_{\RZ} |\nabla_s \kappa_n|^2 \langle \tau, \partial_s V \rangle ds
 $$
 where
 $$
  \delta_V (\nabla_s \kappa) =(\delta_V (\nabla_s \kappa))^\bot = (\partial_s^3 V)^\bot + \partial_s^2 V \ast \tau \ast \kappa   + \partial_s V \ast (\tau  \ast \partial_s \kappa + \kappa \ast \kappa + \kappa \ast \kappa \ast \tau).
 $$
 Hence, 
 \begin{align*}
 \varepsilon | \delta_V \F (\gn)| &\leq  \varepsilon C (\| \partial ^3_s \gn \|_{L^2} + \|\partial^3_s \gn\|_{L^2} (\|\kappa_n\|^2_{L^4} + \|\kappa_n\|_{L^2})) 
 \\ & \leq \varepsilon  C (\|\partial^3_s \gn\|_{L^2} + \|\partial^3_s \gn \|^{\frac 3 2}_{L^2} +1)
 \leq \varepsilon C (\|\partial^3_s \gn\|_{L^2}^2 +1). 
 \end{align*}
 So we get 
 \begin{equation} \label{eq:ConvergenceF1}
  \varepsilon \int_0^\infty \delta_V \F(\gn) dt \rightarrow 0
 \end{equation}
 as $n$ goes to $\infty$.
 
 To get control of the first variation of $\EPD$, we first observe that after taking a subsequence we can furthermore assume that $\kappa_n$ converges to $\kappa$ almost everywhere in space and time. Using the convexity of $x\rightarrow x^{\frac p2}$ for $p>2$  we see that
 $$
  ( |\kappa_n|^2 + \delta_n^2 )^{\frac p 2} \leq 2 ^{\frac p2 -1} (|\kappa_n|^p + \delta_n^p)
 $$
 So,  $$ ( |\kappa_n|^2 + \delta_n^2 )^{\frac p 2}$$ is uniformly integrable. As it also converges pointwise almost everywhere to $|\kappa|^p$  as $n \rightarrow \infty$, we get by Vitali's theorem that 
$$
   ( |\kappa_n|^2 + \delta_n^2 )^{\frac p 2} \rightarrow |\kappa|^p
$$
in $L^1$ as $n \rightarrow \infty.$ Similarly,
$$
   ( |\kappa_n|^2 + \delta_n^2 )^{\frac p 2-1} \kappa_n \rightarrow |\kappa|^{p-2} \kappa
$$
in $L^1$ as $n$ goes to $\infty$.
From 
\begin{equation*} 
\delta_V \EPD( \g ) = \int_{\RZ} (|\kappa|^2 + \delta^2 )^{\frac{p-2}{2}} \langle \kappa, \delta_V \kappa \rangle ds + \frac 1 p \int_{\RZ} (|\kappa|^2 + \delta^2)^{\frac p 2} \langle \tau, \partial_s V \rangle ds 
\end{equation*}
and
$$
\delta_V \kappa =  (\partial_s^2 V)^\bot - 2 \langle \partial_s V, \tau \rangle \kappa - \langle \partial_s V, \kappa \rangle \tau
$$
we thus deduce that 
\begin{equation} \label{eq:ConvergenceEP1}
 \delta_V \EPD (\gn ) \rightarrow \delta_V \E^{(p)} (\g)
\end{equation}
as $n \rightarrow \infty$. From
$$
 \delta_V \Ell(\gn) = \int_{\R/\Z} \langle \tau_n, \partial_s V \rangle ds
$$
we finally see that 
\begin{equation} \label{eq:ConvergenceL1}
 \delta_V \Ell(\gn)  \rightarrow \delta_V \Ell(\g).
\end{equation}
Using \eqref{eq:ConvergenceF1}, \eqref{eq:ConvergenceEP1}, and \eqref{eq:ConvergenceL1}, we can let $n$ go to infinity in \eqref{eq:WeakEquation1} to obtain
$$
 \int_0^\infty \int_{\RZ}  \langle \partial_t^{\bot} \g, V \rangle dxdt = \int_0^\infty \delta_V ( \E^{(p)}(\g) + \lambda \Ell(\g)) dt =  \int_0^\infty \delta_V \EP(\g) dt
$$
for all test functions $ V\in C^\infty_c (\RZ, \mathbb R^n)$. Furthermore,  the estimates (1), (2), (3), and (4) of Lemma \ref{lem:PropApproxSolutions} together with the fact that $\ER(\g_0)$ converges to $\EP(\g_0)$ as $\varepsilon, \delta \rightarrow 0$, immediately give the estimates mentioned in the theorem.
 \end{proof}

\subsection{Arbitrary initial data}

For an arbitrary initial regular curve$\g_0$ of class $W^{2,p}$ we pick a sequence of smooth regular curves $\gn_0 \in C^\infty$ converging to $\g_0$ in $W^{2,p}$ with $\sup_{n} \EP(\gn_0) \leq 2 \EP(\g)$. Using Theorem \ref{thm:LongTimeExistenceRElasticSmoothData} , we get weak solution $\gn$ of the initial value problem
$$
 \begin{cases}
  	\partial_t^\bot \gn = \nabla_{L^2} \EP (\gn) \\
	\gn (0, \cdot ) = \gn_0
 \end{cases}
$$
such that $\gn(t, \cdot)$ is parametrized with constant speed. 
Furthermore, these solutions satisfy the estimates
$$
 \|\partial_t \gn\|_{L^2((0, \infty) \times \RZ) }  \leq C , \qquad   \|\kappa_n\|_{L^p((0, T), B^{\frac 1 {2p}}_{p, \infty})} 
 \leq C (T+1 ) ,
$$
and
$$
\|\gn_{t_1} - \gn_{t_0}\|_{L^2} \leq C |t_1 - t_0|^{\frac 1 2} \quad \text{ for all } t_1, t_0 \in [0, \infty) \text{ with } |t_1 - t_0| \leq 1,
$$
where $C>0$ only depends on $\EP (\g_0)$ as $\sup_{n}\EP(\gn_0) \leq 2 \EP(\g_0) $.  As in the proof of Theorem \ref{thm:LongTimeExistenceRElasticSmoothData} we can assume that 
$\gn \rightarrow \g$ strongly in $L^p_{loc} ([0,\infty), W^{2,p}(\RZ))$  to $\g$ and $\partial_t \gn$ converges weakly locally in $L^2([0,\infty), L^2(\RZ))$
to $\partial_t \g$. 
From the evolution equation we see that for all test functions $ V\in C^\infty_c (\RZ, \mathbb R^n)$ we have
\begin{equation} \label{eq:EvolutionEquationApprox}
 \int_{0}^\infty \int_{\RZ} \langle \partial_t^{\bot}  \gn, V \rangle dxdt = \int_0^\infty ( \delta_V\E^{(p)}(\gn)
 + \lambda \delta_V \Ell(\gn) ) dt.
\end{equation}
As $\partial_t \gn$ converges weakly to $\partial_t \g$ in $L^2$, $\tau_n$ converges strongly to $\tau$ in $L^\infty$, and $\partial_t ^\bot \gn = \partial_t \gn - \langle \partial_t \gn, \tau_n \rangle \tau_n$, we get
$$
 \int_0^\infty \langle \partial_t^\bot \gn, V \rangle dx dt \rightarrow \int_0^\infty \langle \partial_t^\bot \g, V \rangle dx dt
 $$
 as $n$ goes to $\infty.$
From 
\begin{equation*} 
\delta_V \E^p ( \g ) = \int_{\RZ} |\kappa|^{p-2} \langle \kappa, \delta_V \kappa \rangle ds + \frac 1 p \int_{\RZ} |\kappa|^p \langle \tau, \partial_s V \rangle ds 
\end{equation*}
and
$$
\delta_V \kappa =  (\partial_s^2 V)^\bot - 2 \langle \partial_s V, \tau \rangle \kappa - \langle \partial_s V, \kappa \rangle \tau 
$$
we thus deduce that 
$$
 \delta_V \EP (\gn ) \rightarrow \delta_V \E^{(p)} (\gn)
$$
as $n \rightarrow \infty$. From
$$
 \delta_V \Ell(\g) = \int_{\RZ} \langle \tau, \partial_s V \rangle ds
$$
we finally see that 
$$
 \delta_V \Ell(\gn)  \rightarrow \delta_V \Ell(\g).
$$
Hence, we can let in Equation \eqref{eq:EvolutionEquationApprox}  $n$ go to infinity to get
$$
\int_{\RZ} \langle \partial_t^{\bot} \g, V \rangle = \int_0^\infty ( \delta_V\E^{(p)}(\g)
 + \lambda \delta_V \Ell(\g) ) dt
$$
for all test functions $ V\in C^\infty_c (\RZ, \mathbb R^n)$. Furthermore, we have 
$$
 \|\partial_t \g\|_{L^2((0, \infty) \times \RZ) }  \leq C \quad \textnormal{ and } \quad   \|\kappa\|_{L^p((0, T), B^{\frac 1 {2p}}_{p, \infty})} 
 \leq C (T+1). 
$$

\section{Asymptotics of the solution}

We use the last estimates to choose $t_n \rightarrow \infty$ such that both $\|\partial_t \g\|_{L^2([t_n, t_n+1] \times \RZ)} \rightarrow 0$ and $\|\kappa\|_{L^p([t_n, t_n +1 ], B^{\frac 1 {2p}}_{p, \infty}}$ is uniformly bounded. 
We now set
$$
 \gn:[0,1] \times \RZ \rightarrow \mathbb R^n, \quad \gn(t,x) = \g(t_n +t,x) - \g(t_n,0).
$$
Again we can assume after going to a subsequence, that  
$\gn \rightarrow \g$ strongly in $L^p([0,1], W^{2.p}(\RZ))$. From the evolution equation we see that for all test functions $ V\in C^\infty_c ((0,1)\times\RZ,\R^n)$ we have
\begin{equation} 
 \int_0^1\int_{\RZ} \langle \partial_t^{\bot}  \g, V \rangle dx dt= \int_0^1 ( \delta_V\E^{(p)}(\gn)
 + \lambda \delta_V \Ell(\gn) ) dt.
\end{equation}
As $\partial_t \gn$ converges to $0$ in $L^2$, $\tau_n$ is bounded, and $\partial_t ^\bot \gn = \partial_t \gn - \langle \partial_t \gn, \tau_n \rangle \tau_n$, we get
$$
 \int_0^1 \int_{\RZ} \langle \partial_t^\bot \gn, V \rangle dx dt \rightarrow 0
$$
 as $n$ goes to $\infty.$
From 
\begin{equation*} 
\delta_V \E^p ( \g ) = \int_{\RZ} |\kappa|^{p-2} \langle \kappa, \delta_V \kappa \rangle ds + \frac 1 p \int_{\RZ} |\kappa|^p \langle \tau, \partial_s V \rangle ds 
\end{equation*}
and
$$
\delta_V \kappa =  (\partial_s^2 V)^\bot - 2 \langle \partial_s V, \tau \rangle \kappa - \langle \partial_s V, \kappa \rangle \tau 
$$
we deduce that 
$$
 \delta_V \EP (\gn ) \rightarrow \delta_V \E^{(p)} (\gn)
$$
as $n \rightarrow \infty$. From
$$
 \delta_V \Ell(\g) = \int_{\RZ} \langle \tau, \partial_s V \rangle ds
$$
we get 
$$
 \delta_V \Ell(\gn)  \rightarrow \delta_V \Ell(\g).
$$
Hence, we can let  $n$ go to infinity  in Equation \eqref{eq:EvolutionEquationApprox} to get
\begin{equation} \label{eq:Stationary}
0 =  \int_0^1( \delta_V\E^{(p)}(\g)
 + \lambda \delta_V \Ell(\g) ) dt
\end{equation}
for all test functions $ V\in C^\infty_c ((0,1) \times \RZ, \mathbb R^n)$.

For $\tilde V \in C^\infty (\RZ, \mathbb R^n )$  and a non-vanishing $\phi \in C^\infty_c ((0,1), [0,1])$ we use Equation \eqref{eq:Stationary} with $V(t,x) = \phi(t)\tilde V(x) $ to get
$$
 0 = \int_0^1\phi dt ( \delta_{\tilde V}\E^{(p)}(\g)
 + \lambda \delta_{\tilde V} \Ell(\g) ).
$$
Hence we finally obtain
$$
 \delta_{\tilde V}\E^{(p)}(\g)
 + \lambda \delta_{\tilde V} \Ell(\g) ) = 0 
$$
for all $\tilde V \in C^\infty (\RZ, \mathbb R^n )$,  as  $\int_0^1\phi dt>0$. So $\g$ is a critical point of $\EP$. This finishes the proof of Theorem \ref{thm:LongTimeExistenceRElastic}

\bibliographystyle{alpha}
\bibliography{Master}

\end{document}